\documentclass[11pt]{article}
\usepackage{xcolor}
\definecolor{aleacolour}{rgb}{0.09,0.32,0.44} 
\usepackage[pdfborder={0 0 0},pdfborderstyle={},colorlinks=true,linkcolor=aleacolour,citecolor=aleacolour,urlcolor=blue]{hyperref}
\usepackage[textsize=scriptsize,colorinlistoftodos]{todonotes}

\usepackage{amsthm,amsmath,amssymb,amsfonts}
\usepackage{enumerate}
\usepackage{cases}
\usepackage[small,bf]{caption}
\usepackage[top=2.1cm,bottom=2.2cm,left=2cm,right=2cm]{geometry}
\usepackage[czech, english]{babel}
\usepackage[nolabel, final]{showlabels}

\newcommand{\be}{\begin{enumerate}}

\newcommand{\beq}{\begin{equation}}
\newcommand{\eeq}{\end{equation}}
\newcommand{\beqs}{\begin{equation*}}
\newcommand{\eeqs}{\end{equation*}}
\newcommand{\bea}{\begin{eqnarray}}
\newcommand{\eea}{\end{eqnarray}}
\newcommand{\beas}{\begin{eqnarray*}}
\newcommand{\eeas}{\end{eqnarray*}}

\def\({\left(}
\def\){\right)}

\theoremstyle{plain}
\newtheorem{theorem}{Theorem}

\newtheorem{lemma}[theorem]{Lemma}
\newtheorem{claim}[theorem]{Claim}

\newtheorem{conjecture}[theorem]{Conjecture}

\theoremstyle{definition}
\newtheorem{definition}[theorem]{Definition}

\newtheorem{sclaim}{Claim}

\theoremstyle{remark}

\title{\vspace{-0.9cm}Threshold Ramsey multiplicity for paths and even cycles}
\author{David Conlon\thanks{Department of Mathematics, California Institute of Technology, Pasadena, CA 91125. Email:  \href{dconlon@caltech.edu} {\nolinkurl{dconlon@caltech.edu}}. Research supported by NSF Award DMS-2054452.}
\and 
Jacob Fox\thanks{Department of Mathematics, Stanford University, Stanford, CA 94305. Email: \href{jacobfox@stanford.edu} {\nolinkurl{jacobfox@stanford.edu}}. Research supported by a Packard Fellowship and by NSF Award DMS-1855635.}
\and 
Benny Sudakov\thanks{Department of Mathematics, ETH, 8092 Z\"urich, Switzerland. Email: \href{mailto:benjamin.sudakov@math.ethz.ch} {\nolinkurl{benjamin.sudakov@math.ethz.ch}}. Research supported by SNSF Grant 200021\_196965.}
\and
Fan Wei\thanks{Department of Mathematics, Princeton University, Princeton, NJ 08540. Email: \href{mailto:fanw@princeton.edu} {\nolinkurl{fanw@princeton.edu}}. Research supported by NSF Award DMS-1953958.}
}
\date{}

\begin{document}

\maketitle

\begin{abstract}
The Ramsey number $r(H)$ of a graph $H$ is the minimum integer $n$ such that any two-coloring of the edges of the complete graph $K_n$ contains a monochromatic copy of $H$. While this definition only asks for a single monochromatic copy of $H$, it is often the case that every two-edge-coloring of the complete graph on $r(H)$ vertices contains many monochromatic copies of $H$. The minimum number of such copies over all two-colorings of $K_{r(H)}$ will be referred to as the threshold Ramsey multiplicity of $H$. Addressing a problem of Harary and Prins, who were the first to systematically study this quantity, we show that there is a positive constant $c$ such that the threshold Ramsey multiplicity of a path or an even cycle on $k$ vertices is at least $(ck)^k$. This bound is tight up to the constant $c$. We prove a similar result for odd cycles in a companion paper.
\end{abstract}

\section{Introduction}

The \emph{Ramsey number} $r(H)$ of a graph $H$ is  the minimum positive integer $n$ such that any two-coloring of the edges of the complete graph $K_n$ on $n$ vertices contains a monochromatic copy of $H$. Determining Ramsey numbers is a challenging task and the exact value of $r(H)$ is known in only a few special cases. For example, determining the Ramsey number of $K_5$, the complete graph on five vertices, is a well-known open problem. 

The few non-trivial families for which the Ramsey number is known exactly include paths and cycles. To say more, we let $P_k$ denote the path on $k$ vertices and $C_k$ the cycle on $k$ vertices. The length of a path or cycle denotes its number of edges, so $P_k$ has length $k-1$ and $C_k$ has length $k$. In 1967, Gerencs\'er and Gy\'arf\'as~\cite{GG} determined the Ramsey number of paths, showing that
\[r(P_k) = k  - 1+ \lfloor k/2 \rfloor.\]
It is simple to show that $r(C_3) = r(C_4) = 6$, while, for $k \geq 5$, Faudree and Schelp \cite{cycle2} and, independently, Rosta \cite{Ros} proved that
 \[ r(C_k) = k + k/2- 1 \text{ if } k \text{ is even \ and} \ \ r(C_k)=2k-1  \text{ if } k \text{ is odd}.\] 

A more general problem is to determine the \emph{Ramsey multiplicity} $M(H,n)$, defined to be the minimum number of monochromatic copies of $H$ that appear in any two-edge-coloring of $K_n$. In particular, $M(H, n) = 0$ if and only if $n < r(H)$, so the problem of determining $M(H, n)$ does indeed generalize the problem of determining $r(H)$. 

In 1962, Erd\H{o}s~\cite{Erdos62} conjectured that if $H$ is a clique, then $M(H, n)$ is asymptotically equal to the expected number of monochromatic copies of $H$ in a uniformly random two-coloring of the edges of $K_n$ and Burr and Rosta~\cite{BRcommon} later generalized this conjecture to all graphs $H$. While true for $K_3$, a result of Goodman~\cite{Goodman} that predates the conjecture and doubtless inspired it, Thomason~\cite{T89} showed that it is already false for $K_4$. Despite the failure of this attractive conjecture, the asymptotic behavior of $M(H,n)$ for fixed $H$ and $n$ tending to infinity has drawn considerable attention (see~\cite{C12, Fox, GLLV, HHKNR, K4common} for some examples or~\cite{Ramseysurvey} for a survey). In particular, it is known that any bipartite graph which satisfies the well-known conjecture of Erd\H{o}s--Simonovits \cite{Simon} and Sidorenko \cite{Sidorenko,Sidorenko2} also satisfies the Burr--Rosta conjecture, so the considerable recent progress~\cite{CFS,CKLL,CL,CL21,Hatami,KLL,LS,S1} on Sidorenko's conjecture, as it is usually known, may also be interpreted as progress on our understanding of Ramsey multiplicity.

Besides the case where $H$ is fixed and $n$ tends to infinity, another much-studied problem asks for the value of $M(H, n)$ when it first becomes positive, that is, when $n = r(H)$. To distinguish it from the more general Ramsey multiplicity function $M(H, n)$, we call this  value the \emph{threshold Ramsey multiplicity}. 

\begin{definition}[Threshold Ramsey multiplicity]
The \emph{threshold Ramsey multiplicity} $m(H)$ of a graph $H$ is the minimum number of monochromatic copies of $H$ in any two-coloring of the edges of  $K_n$ with $n=r(H)$. In other words, \[m(H) = M(H, r(H)).\]
\end{definition}

The threshold Ramsey multiplicity was first studied systematically by Harary and Prins~\cite{HP} almost fifty years ago. Their work and subsequent work by Schwenk \cite{Hsurvey} and by Piwakowski and  Radziszowski~\cite{K4} determines the threshold Ramsey multiplicity for all graphs with at most four vertices. However, in general, the problem of determining or even giving a non-trivial lower bound on the threshold Ramsey multiplicity appears extremely difficult. This is in part because it seems necessary to first determine the Ramsey number, a problem which is already hard, before one can say anything substantive about the threshold Ramsey multiplicity.   

The only family for which $m(H)$ is known is for stars, where Harary and Prins~\cite{HP} proved that $m(K_2) = 1$ and  $m(K_{1,k}) =1$ for $k$ even, but $m(K_{1,k}) =2k$ for $k \geq 3$ odd, surprisingly erratic behavior for such a simple family. In the same paper, Harary and Prins asked for a determination of the threshold Ramsey multiplicity for paths and cycles, probably the next simplest families after stars. The main result of this paper is an approximate answer to their question for paths and even cycles. The case of odd cycles will be discussed in the companion paper~\cite{CFSW2}.  

To the best of our knowledge, the only previous work concerning these questions is due to Rosta and her collaborators, who looked at the case of odd cycles. In her first paper on the subject, with Sur\'anyi~\cite{RStm}, she obtained the  exponential lower bound $m(C_k)\geq 2^{c k}$. This was later improved to a superexponential lower bound in an unpublished work. More recently, K\'arolyi and Rosta \cite{KR} improved the lower bound to $m(C_k) \geq k^{ck}$, which we will see below is sharp up to the constant in the exponent. However, their method has little to say about paths and even cycles, the main objects of interest in this paper, in large part because the Ramsey numbers of these graphs are significantly smaller than the Ramsey number of odd cycles of comparable size.

Our main result, proved in this paper and its companion~\cite{CFSW2}, is the following.

\begin{theorem}\label{thm:main}
There is a positive constant $c$ such that, for every positive integer $k$, the threshold Ramsey multiplicity of paths and cycles on $k$ vertices satisfy $m(P_k) \geq (ck)^k$ and $m(C_k) \geq (ck)^k$.
\end{theorem}

We prove Theorem \ref{thm:main} for paths and even cycles in this paper, while the case of odd cycles is handled in the companion paper~\cite{CFSW2}. 

The bound in Theorem \ref{thm:main} is easily seen to be tight up to the constant $c$, since the total number of paths or cycles on $k$ vertices in the complete graph with
$r(P_k)$ or $r(C_k)$ 
vertices is at most $(c'k)^k$ for some constant $c'$. 
However, we may also pinpoint some edge-colorings which we believe to be optimal for $m(P_k)$ and $m(C_k)$. 
Consider the edge-coloring $\chi(a,b)$ of the complete graph on $n=a+b$ vertices with vertex set $A \cup B$, $|A|=a$ and $|B|=b$, where $A$ and $B$ form blue cliques and all edges between $A$ and $B$ are red. Let $a_0=k-1$ and $b_0=\lfloor k/2 \rfloor -1$. The coloring $\chi(a_0,b_0)$ does not contain a monochromatic $P_k$ and gives the tight lower bound for the Ramsey number of the path $P_k$.  
If $k$ is even, the colorings $\chi(a_0+1,b_0)$ and $\chi(a_0,b_0+1)$ of the complete graph on $a_0+b_0+1=r(P_k)$ vertices each have exactly $k!/2$ monochromatic $P_k$. If $k$ is odd, the coloring $\chi(a_0,b_0+1)$ of the complete graph on $a_0+b_0+1=r(P_k)$ vertices has exactly $\frac{(k-1)}{4}(k-1)!$ monochromatic $P_k$. Not only do these colorings show that Theorem~\ref{thm:main} is tight up to the constant $c$ for paths, but we conjecture that they realize the threshold Ramsey multiplicity for $k$ sufficiently large.

 \begin{conjecture} 
 For sufficiently large $k$, if $k$ is even, then $m(P_k) = k!/2$ and if $k$ is odd, then $m(P_k) = \frac{(k-1)}{4} (k-1)!$.
\end{conjecture}

As $P_k$ is a subgraph of $C_k$, the edge-coloring $\chi(a_0,b_0)$ with $a_0=k-1$ and $b_0=\lfloor k/2 \rfloor -1$ described above also does not contain a monochromatic $C_k$. For $k \geq 6$ even, this coloring realizes the tight lower bound on $r(C_k)$. The coloring formed from $\chi(a_0+1,b_0)$ by changing the color of one edge in the monochromatic blue clique of order $a_0+1=k$ to red does not have a monochromatic red $C_k$  and thus has $(k-1)!/2 - (k-2)! = \frac{(k-3)}{2}(k-2)!$ monochromatic $C_k$. We conjecture that for $k$ sufficiently large this is the threshold Ramsey multiplicity for the even cycle $C_k$.  

If $k$ is odd, then the coloring $\chi(k-1,k-1)$ has no monochromatic $C_k$ and realizes the tight lower bound on the Ramsey number $r(C_k)$. In this case, the edge-coloring $\chi(k,k-1)$ has all monochromatic $C_k$ in the blue clique of order $k$ and thus has $(k-1)!/2$ monochromatic $C_k$. We conjecture that for $k$ sufficiently large this is the threshold Ramsey multiplicity for the odd cycle $C_k$.  

\begin{conjecture} For sufficiently large $k$, if $k$ is even, then $m(C_k) =  \frac{(k-3)}{2} (k-2)!$ and if $k$ is odd, then $m(C_k) = (k-1)!/2$.
\end{conjecture}

The rest of the paper is dedicated to the proof of Theorem~\ref{thm:main} in the case of paths and even cycles. Because we focus entirely on this case, we will often use the phrase Theorem~\ref{thm:main} as a shorthand to mean Theorem~\ref{thm:main} for paths and even cycles.
We note that we have made no attempt to optimize the value of the constant $c$ in Theorem~\ref{thm:main}. Throughout the proof, we have also chosen to omit floor and ceiling signs whenever they are not essential.

\section{Proof of Theorem \ref{thm:main} for paths and even cycles}\label{sec:even}

Szemer\'edi's regularity lemma (see Lemma \ref{reglem} below) will be an important tool in our proof.  Given any graph, the regularity lemma shows that there is a vertex partition of the graph into a small number of parts of almost equal size, where the bipartite graph between almost every pair of parts is random-like. This property is useful for many purposes, particularly for embedding and counting sparse subgraphs. For an excellent (though now somewhat outdated) survey, we refer the interested reader to~\cite{Regularity}.

To state the regularity lemma, we need some definitions making precise what is meant by saying that the graph between two vertex sets is ``random-like". For a pair of vertex subsets $(X,Y)$ of a graph, let $e(X,Y)$ denote the number of pairs in $X \times Y$ that are edges and $d(X, Y) = e(X,Y)/|X||Y|$ denote the density of edges between $X$ and $Y$. 

\begin{definition}[$\epsilon$-regular pair]
A pair of vertex subsets $(X, Y)$ of a graph is \emph{$\epsilon$-regular} if, for all subsets $U \subset X, V \subset Y$ such that $|U| \geq \epsilon |X|$ and $|V| \geq \epsilon |Y|$, $|d(U, V) - d(X,Y)| \leq \epsilon$. 
\end{definition}

The following lemma collects some basic facts which follow easily from this definition. 

\begin{lemma} \label{lem:epsregprop}
If $(X,Y)$ is an $\epsilon$-regular pair in a graph $G$ and $d=d(X,Y)$, then the following hold:
\begin{enumerate}[(i)]
\item If $Y' \subset Y$ satisfies $|Y'| \geq \epsilon |Y|$, then the number of vertices in $X$ with degree in $Y'$ greater than $(d+\epsilon)|Y'|$ is less than $\epsilon |X|$ and the number of vertices in $X$ with degree in $Y'$ less than $(d-\epsilon)|Y'|$ is less than $\epsilon |X|$. 
\item If $X' \subset X$ and $Y'\subset Y$ are such that $|X'| \geq \alpha |X|$ and $|Y'| \geq \alpha |Y|$, then $(X', Y')$ is $\max(\epsilon/\alpha, 2\epsilon)$-regular.
\item Provided $X$ and $Y$ are disjoint, the pair $(X,Y)$ is also $\epsilon$-regular in the complement of $G$.
\end{enumerate}
\end{lemma}

A partition of a set is said to be {\it equitable} if each pair of parts differ in size by at most one. With this definition, we can now state the regularity lemma in a standard colored form, whose equivalence to the usual form follows easily from Lemma~\ref{lem:epsregprop}(iii). 

\begin{lemma}[Szemer\'edi's regularity lemma]\label{reglem} 
For every $\epsilon>0$ and positive integer $m_0$, there exist positive integers $M_0$ and $n_0$ such that every two-edge-coloring of the complete graph $K_n$ with $n \geq n_0$ in colors red and blue 
admits an equitable vertex partition 
$V_1 \cup \cdots \cup V_M$ into $M$ parts with $m_0 \leq M \leq M_0$ where all but at most $\epsilon{M \choose 2}$ pairs $(V_i,V_j)$ of parts with $1 \leq i < j \leq M$ are $\epsilon$-regular in both the red and  blue subgraphs. 
\end{lemma}

We remark that there is a strengthening of the regularity lemma, proved in \cite{CFW}, where each part is $\epsilon$-regular with all but an $\epsilon$-fraction of the other parts and each part is also $\epsilon$-regular with itself. Working with this variant rather than Lemma~\ref{reglem} would allow us to simplify our proof very slightly. However, since this variant is, as yet, non-standard, we have opted to work with the usual version instead.

Once we have the partition guaranteed by the regularity lemma, it is often convenient to consider a simplified rendering of the graph 
called the \emph{reduced graph} of the partition. By saying that a graph is \emph{red/blue-multicolored}, we will mean that each edge is colored either blue, red, or both blue and red.

\begin{definition}
Given a red/blue-edge-colored graph $G$, a partition $V_1 \cup \dots \cup V_M$ of its vertex set
and parameters $0 < \epsilon, d < 1$, 
the \emph{reduced graph} $H=H(\epsilon,d)$ of the partition with parameters $\epsilon$ and $d$ is the red/blue-multicolored graph with vertex set $[M]$ and a red (respectively, blue) edge between $i$ and $j$ if and only if $(V_i, V_j)$ is $\epsilon$-regular with density at least $d$ in the red (respectively, blue) graph.
\end{definition}

With this preliminary, we may now give a broad outline of the proof of Theorem~\ref{thm:main}.

\subsection{Proof Outline}

We first prove that in any red/blue edge-coloring of the complete graph $K_n$, there is a color and an almost spanning subset $W$ of the vertices such that, for any two vertices of $W$, there are many short paths between them in the specified color. We then apply Szemer\'edi's regularity lemma to the subgraph of $K_n$ induced by $W$, obtaining a reduced graph. If, in this reduced graph, we can find a large monochromatic matching, then we can build as many of the required paths and even cycles as we need. This case will be discussed in detail in Section~\ref{subsec:matchingimply}.

 If, instead, there is no sufficiently large monochromatic matching in the reduced graph, then a key stability result (Lemma \ref{lem:main} below) shows that the original two-colored graph $G$ induced by the vertex set $W$ is close to a certain shape (described in Definition \ref{def:extcoloring}). In this case,  we can directly bound the number of paths and even cycles to complete the proof. The details of this case may be found in Section~\ref{subsec:case2implypath}.

A variant of our stability lemma already appeared in the work of Gy\'arf\'as, S\'ark\"ozy, and Szemer\'edi~\cite{stability}. However, the version we need is somewhat stronger, so we include a complete proof in Section~\ref{subsec:stability1}. One point worth noting is that we make an appeal to the regularity lemma in our statement and proof, whereas the stability lemma in~\cite{stability} is proved without it. 
We now describe our version in more detail.

\subsection{The stability lemma}\label{subsec:stability}

The next two definitions already appear in the work of Gy\'arf\'as, S\'ark\"ozy, and Szemer\'edi~\cite{stability}, though the first is stated in slightly more generality than in~\cite{stability}.

\begin{definition}[Well-connected] \label{def:EC1}
A vertex subset $W$ of a graph $G$ is \emph{$(t,l)$-well-connected} if any two vertices $u, v \in W$ are connected by at least $t$ internally vertex-disjoint paths of length at most $l$. Note that any vertex in $V(G) \setminus \{u,v\}$ is allowed as an internal vertex for these paths.
\end{definition}

We will often refer to a vertex set as being {\it well-connected in a particular color}, meaning that the vertex set is well-connected with respect to the graph consisting of edges in that color. For the second definition, given a graph $G$ and disjoint vertex subsets $A$ and $B$, we let $G[A]$ denote the induced subgraph of $G$ with vertex set $A$ and $G[A,B]$ the bipartite graph with parts $A$ and $B$ whose edges are the edges of $G$ between $A$ and $B$. Note that the density within a set $X$ is given by $d(X,X) = e(X,X)/|X|^2 = 2e(X)/|X|^2$. 

\begin{definition}[Extremal coloring with parameter $\alpha$]\label{def:extcoloring} A two-coloring of the edges of a graph $G$ is an \emph{extremal coloring with parameter $\alpha$} if there exists a partition $V(G) = A \cup B$ such that 
\begin{itemize}
\item $|A| \geq (2/3 - \alpha)|V(G)|$ and $|B| \geq (1/3 - \alpha)|V(G)|$ and 
\item the  graph $G[A]$ has density at least $(1-\alpha)$ in some color and the bipartite graph $G[A,B]$ has density at least $(1-\alpha)$ in the other color. 
\end{itemize}
\end{definition}

Our key stability lemma is now as follows. Roughly speaking, it says that every two-coloring of the edges of $K_n$ is either close to an extremal coloring or the reduced graph contains a monochromatic matching covering more than $2/3$ of the vertices such that the underlying vertex set is well-connected in the same color.

\begin{lemma}\label{lem:main}
For any $0< \epsilon \leq 10^{-10}$ and $d,\lambda  \geq 1000\epsilon$, there is a positive integer $M_0 = M_0(\epsilon)$ such that if $n$ is sufficiently large in terms of $\epsilon$, then any two-coloring of the edges of the complete graph $K_n$ falls into at least one of the following two cases:

\begin{itemize}
\item 
\noindent \textbf{Case 1:} There is a positive integer $M \leq M_0$ and disjoint vertex subsets $U_1, \dots, U_m, V_1, \dots, V_m$ with $m \geq (2/3 +\lambda)M/2$ such that each $|U_i|, |V_i| \geq cn$ with $c \geq (1-\epsilon)/M$, all pairs $(U_i, V_i)$ are simultaneously $\epsilon$-regular in some color with the edge density in that color at least $d-\epsilon$, and $\bigcup_{i=1}^m U_i \cup \bigcup_{i=1}^m V_i$ is $(200M, 6)$-well-connected in the same color.

\item
\noindent \textbf{Case 2:} The coloring is an extremal coloring with parameter $1000(d+\lambda+\sqrt{\epsilon})$. 
\end{itemize}
\end{lemma}

Observe that if $\alpha \geq 2/3$, any two-coloring of the edges of a complete graph is trivially an extremal coloring with parameter $\alpha$, since we may take $A$ to be the empty set. It follows that we may assume $d,\lambda \leq 1/1000$ in Lemma \ref{lem:main}.

\subsection{Proof of Theorem \ref{thm:main} assuming Lemma \ref{lem:main}}
We now prove Theorem \ref{thm:main} by applying Lemma \ref{lem:main} with $d = 20\sqrt{\epsilon}$ and $\lambda = 13\sqrt{\epsilon}$. 

\subsubsection{Proof of Theorem \ref{thm:main} in the situation of Case 1 of Lemma \ref{lem:main}}\label{subsec:matchingimply}

We first prove Theorem \ref{thm:main} for paths for edge-colorings satisfying Case 1 of Lemma \ref{lem:main} with the following approach. Roughly speaking, in the graph of the color given in this case, between any regular pair $(U_i, V_i)$ with density $d(U_i, V_i)= d$, there should be many paths of length close to $2cn$. Since the bipartite graph between $U_i$ and $V_i$ is random-like, the count of paths of length $l$ is roughly at least $d^l\prod_{i=0}^{l} (cn-\lfloor i/2 \rfloor)$. Since the union of the $U_i$ and $V_i$ is well-connected, any two vertices in this union are connected by many internally vertex-disjoint short paths. We can then find many long paths $P_k$ by using the short paths guaranteed by the well-connectedness property to connect the end vertices of the paths from different pairs $(U_i, V_i)$. In this section, we will make this idea rigorous. 

The following two lemmas show that for a regular pair $(U,V)$ in a graph $G$ the number of long paths starting from any vertex of large degree or between any pair of vertices of large degree in the bipartite graph $G[U,V]$ is roughly at least the expected count if $G[U,V]$ were a random graph of the same density. 

\begin{lemma}\label{lem:countpath}
Suppose $(U,V)$ is an $\epsilon$-regular pair of disjoint vertex subsets of a graph $G$ such that $|U|,|V| \geq n$ and $d(U,V) = d$. 
If $n \geq \epsilon^{-2}$ and $d  > \epsilon + \sqrt{\epsilon}$, then, for any vertex $v \in V$ with at least $(d-\epsilon)|U|$ neighbors in $U$ and any positive integer $l \leq 2(1- \sqrt{\epsilon}) n -1$, there are at least $(d-\epsilon - \sqrt{\epsilon})^l \prod_{i=1}^l(n-\lfloor i/2 \rfloor)$ paths of length $l$ in $G[U,V]$ starting from $v$.
\end{lemma}

\begin{proof}
Let $N_j$ be the number of paths $P$ of length $j$ in $G(U,V)$ of the form $v_0=v, v_1, \dots, v_j$ starting from $v$ for which there are at least $(d-\epsilon)(|U| - \lfloor (j+1)/2 \rfloor)$ ways to extend the path if $j$ is even and at least $(d-\epsilon)(|V| - \lfloor (j+1)/2 \rfloor)$ ways to extend the path if $j$ is odd. By extending the path, we mean finding a vertex $v_{j+1}$ that is adjacent to $v_j$ but distinct from the vertices in $P$. We will prove by induction on $j$ that for $j \leq 2(1- \sqrt{\epsilon}) n-2$, we have $N_j \geq (d-\epsilon - \sqrt{\epsilon})^j \prod_{i=1}^j (n-\lfloor i/2 \rfloor)$, which easily implies the lemma.

Clearly $N_0 = 1$, since a path with zero edges starting from $v$ is just $v$ itself and it is extendable in sufficiently many ways by the degree condition on $v$. This is the base case of the induction.

Suppose now that we have the claimed lower bound on $N_j$ for some $j \leq 2(1-\sqrt{\epsilon})n-3$ and we wish to prove the lower bound on $N_{j+1}$. Suppose $j$ is even (the case where $j$ is odd can be handled in exactly the same way). Let $P:v_0=v,\ldots,v_j$ be a path in $G(U,V)$ of length $j$ which can be extended in at least $(d-\epsilon)(|U| - \lfloor (j+1)/2 \rfloor)$ ways. Then $v_j \in V$ and there are at least $(d-\epsilon)(|U| - \lfloor (j+1)/2 \rfloor)$ neighbors of $v_j$ in $U$ which are not in $P$. We let $U'$ be this set of neighbors. As the pair $(U,V)$ is $\epsilon$-regular and $|V \setminus \{v_0,  v_1, \dots, v_{j}\}| = |V| - \lfloor (j+2)/2 \rfloor \geq \epsilon |V|$, Lemma \ref{lem:epsregprop}(i) implies that there are fewer than $\epsilon |U|$ vertices in $U$ whose degree in  $V \setminus \{v_0,  v_1, \dots, v_{j}\}$ is less than $(d-\epsilon)|V \setminus \{v_0,  v_1, \dots, v_{j}\}|=(d-\epsilon)|(|V|-\lfloor (j+2)/2 \rfloor)$. 
Therefore, the number of vertices in $U'$ which can be used as $v_{j+1}$ and added to $P$ so that this longer path is extendable in sufficiently many ways is at least 
\[ |U'|-\epsilon |U| \geq (d-\epsilon)(|U| - \lfloor (j+1)/2 \rfloor)-\epsilon |U| \geq (d-\epsilon)(n - \lfloor (j+1)/2 \rfloor)-\epsilon n  \geq (d-\epsilon-\sqrt{\epsilon})(n- \lfloor (j+1)/2 \rfloor),\] 
where the final inequality follows from the upper bound on $j$ assumed in the lemma. Hence, $N_{j+1} \geq N_j (d-\epsilon - \sqrt{\epsilon})(n - \lfloor (j+1)/2 \rfloor)$. By the lower bound on $N_j$, we obtain the desired lower bound on $N_{j+1}$, completing the induction. 
\end{proof}

\begin{lemma}\label{lem:countpath1}
Suppose $(U,V)$ is an $\epsilon$-regular pair of disjoint vertex subsets of a graph $G$ such that $|U|,|V| \geq n$ with $n \geq 5\epsilon^{-2}$ and $d(U,V) = d$ with $d  > 5\sqrt{\epsilon}$. Let $u,v \in U \cup V$ be distinct vertices which are each adjacent to at least a $(d-\epsilon)$ fraction of the vertices in the other part. Suppose $l$ is an integer with $3 \leq l \leq 2(1-2\sqrt{\epsilon})n$, where $l$ is even if $u$ and $v$ are in the same part and $l$ is odd if $u$ and $v$ are in different parts. Then the number of paths of length $l$ in $G[U,V]$ with end vertices $u$ and $v$ is at least $(d-7 \sqrt{\epsilon})^{l-1}  (\epsilon n) \prod_{i=1}^{l-2} (n-\lfloor i/2 \rfloor)$.
\end{lemma}

\begin{proof}
We will focus on the case where $u\in U$ and $v\in V$. The case where $u$ and $v$ are in the same part can be handled similarly.

As $|N(u)|\geq (d-\epsilon)|V| \geq \epsilon |V|+1$, we can set aside $\epsilon |V|$ neighbors of $u$ (not including $v$) and remove them from $V$, calling this set of $\epsilon |V|$ neighbors $V_0$. We will only use these vertices in the last step to connect with $u$. 
As $1 \geq d >5\sqrt{\epsilon}$, we have $\epsilon<1/25$. By Lemma \ref{lem:epsregprop}(ii) with $\alpha=1-\epsilon$, and noting that $\max(\epsilon/(1-\epsilon), 2\epsilon) = 2\epsilon$, the pair $(V \setminus V_0,U \setminus \{u\})$ is $ 2\epsilon$-regular. 

Let $l$ be an odd positive integer. Our aim is to give a lower bound on the number of paths of length $l$ with end vertices $u$ and $v$. 
Suppose that we fix a path of length $l-3$ starting from $v$, say $P: w_0 = v, w_1, \dots, w_{l-3}$, such that the vertices are in $(V \setminus V_0) \cup (U\setminus \{u\})$ and there are at least  $(d-2\epsilon)(|U|-1- \lfloor (l-2)/2 \rfloor)$ ways to extend the path to a vertex $w_{l-2} \in U$. Let $W_P$ be this set of candidate vertices for $w_{l-2}$. Then 
\begin{align*}  |W_P| \geq & \ (d-2\epsilon)(|U|-1- \lfloor (l-2)/2 \rfloor)\geq (d-2\epsilon) (|U|-1- (2(1-2\sqrt{\epsilon})n-2)/2)  \\
 \geq & \   (d-2\epsilon) (|U|- (1-2\sqrt{\epsilon})|U|) = (d-2\epsilon) (2\sqrt{\epsilon}|U|)
\geq \epsilon |U|.
\end{align*}
As $(U,V)$ is $\epsilon$-regular, $|W_P|\geq \epsilon |U|$, and $|V_0| \geq \epsilon |V|$, the number of edges $(w_{l-2},w_{l-1}) \in W_P \times V_0$ satisfies 
\[e(W_P,V_0) \geq (d-\epsilon)|W_P||V_0| > (d-2\epsilon)^2(|U| - \lfloor (l-2)/2 \rfloor)  \cdot (\epsilon |V|).\] 
We can obtain a path of length $l$ from $v$ to $u$ by beginning with the path $P$ of length $l-3$, followed by any pair $(w_{l-2},w_{l-1})$ of adjacent vertices as above, and finally ending with $u$.

For any non-negative integer $i$, let $N_i$ be the total number of paths $P: v_0=v, v_1, \dots,v_i$ of even length $i$ in the bipartite graph $G[V\setminus V_0,U \setminus \{u\}]$ starting from $v$ for which the number of ways to extend the path is at least $(d-2\epsilon)(|U|-1 - \lfloor (i+1)/2 \rfloor)$. 
Applying Lemma~\ref{lem:countpath} with $n$ replaced by $(1-\epsilon)n$, $d$ replaced by $d(V \setminus V_0,U \setminus \{u\}) \geq d-\epsilon$, and $\epsilon$ replaced by $2\epsilon$, we  deduce that the number $N_{l-3}$ of such paths $P$ of length $l-3$ is at least 
\[((d-\epsilon) -2\epsilon- \sqrt{2\epsilon})^{l-3} \prod_{i=1}^{l-3} ((1-\epsilon)n-\lfloor i/2 \rfloor),\]
where we can apply Lemma \ref{lem:countpath} since the conditions on path length, density, and the number of vertices are all satisfied. Therefore, the number of paths of length $l$ with end vertices $u$ and $v$ is at least
\begin{align*}
 & N_{l-3} \cdot (d-2\epsilon)^2(n - \lfloor (l-2)/2 \rfloor)  \cdot (\epsilon n) \\ 
& \geq \ ((d-\epsilon) -2\epsilon- \sqrt{2\epsilon})^{l-3} \prod_{i=1}^{l-3} ((1-\epsilon)n-\lfloor i/2 \rfloor)  \cdot (d-2\epsilon)^2(n - \lfloor (l-2)/2 \rfloor)  \cdot (\epsilon n)\\
& \geq  \ (d-5\sqrt{\epsilon})^{l-1} (1-\epsilon-\sqrt{\epsilon})^{l-3} (\epsilon n) \prod_{i=1}^{l-2} (n-\lfloor i/2 \rfloor) \\
& \geq \ (d- 7 \sqrt{\epsilon})^{l-1} (\epsilon n) \prod_{i=1}^{l-2} (n-\lfloor i/2 \rfloor).
\end{align*}
The second inequality holds since 
$(1-\epsilon) n - \lfloor i/2 \rfloor \geq (1-\epsilon - \sqrt{\epsilon})(n - \lfloor i/2 \rfloor)$ for $i \leq 2(1-2\sqrt{\epsilon})n$. 
\end{proof}

We now prove the path case of Theorem \ref{thm:main}  when the coloring satisfies Case 1 of Lemma \ref{lem:main}.

\begin{proof}[Theorem \ref{thm:main} for paths for colorings satisfying Case 1 of Lemma \ref{lem:main}]
Fix $0<\epsilon \leq 10^{-20}$ and let $d=20\sqrt{\epsilon}$ and $\lambda = 13\sqrt{\epsilon}$. Suppose there are vertex subsets $U_1,\ldots,U_m,V_1,\ldots,V_m$ with $m = (2/3 + \lambda) M/2$ and $|V_i|,|U_i|\geq cn$ satisfying the properties of Case 1 of Lemma \ref{lem:main}, say in color red. We may assume that $n$ is sufficiently large in terms of $c,\epsilon,$ and $M$. 
Let $d'=d-\epsilon$, so the edge density between each pair $(U_i, V_i)$ is at least $d'$. We will show that there is a constant $c'>0$ such that the number of monochromatic paths with $k=\lceil 2(n+1)/3 \rceil$ vertices is at least $(c'k)^{k}$.

We give a lower bound on the number of paths with $k$ vertices in the red graph $G$ by first choosing a pair of anchor vertices $(v_i, u_i) \in V_i \times U_i$ for each $1 \leq i \leq m$ and then picking short paths $P_i$ to join $u_i$ and  $v_{i+1}$ and long paths $T_i$ to join $u_i$ and $v_i$, where we will use Lemma \ref{lem:countpath1} to show that there are many paths $T_i$ connecting $u_i$ and $v_i$ within $G(U_i, V_i)$ that avoid the vertices of all the short paths $P_j$.

\begin{figure}[h]
\centering
\includegraphics[scale=0.5, trim={0cm 0cm 0cm 0},clip]{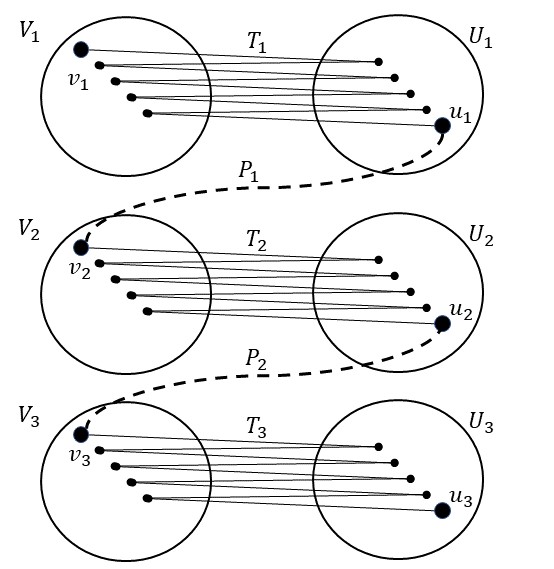}
\caption{An illustration showing the anchor vertices $v_i$ and $u_i$, the short paths $P_i$, and the long paths $T_i$ used to build paths with $k$ vertices.}
\label{fig_create_path}
\end{figure}

From each $V_i$ and $U_i$, pick vertices $v_i \in V_i$ and $u_i \in U_i$ as anchor vertices such that each is adjacent to at least a $(d'-\epsilon)$-fraction of the vertices in the other part.  
Since $(V_i, U_i)$ is $\epsilon$-regular, there are at least $(1-\epsilon)cn$ choices for each of $v_i$ and $u_i$. 

After fixing the choice of pairs of anchor vertices $(v_i, u_i)_{1\leq i \leq m}$, we now pick a set of short disjoint paths $P_i$ to connect $u_i$ to $v_{i+1}$ for each $1 \leq i \leq m-1$. 
By assumption, the vertex set $\bigcup_{i=1}^m V_i \cup \bigcup_{i=1}^m U_i$ is $(200M, 6)$ well-connected. For $1 \leq i \leq m-1$, we will greedily pick a red path $P_i$ of length at most six to connect $u_i$ and $v_{i+1}$. In total, we will pick $m-1$ paths; together with $v_1$ and $u_m$, there will be at most $7(m-1)+2 \leq 7m$ vertices in all the $P_i$'s and anchor vertices. Since there are at least $200M$ internally vertex-disjoint paths of length at most six connecting $u_i$ and $v_{i+1}$ by the definition of a $(200M, 6)$-well-connected set and $200M > 100 (2/3 + \lambda) M/2 = 100m > 7m$, 
we can greedily choose these $m-1$ paths such that they are vertex disjoint and internally do not use any anchor vertices. 

After fixing $P_i$ for $1\leq i \leq m-1$, we will use long paths $T_i$ to connect each pair $(v_i, u_i)$. 
In each regular pair $(V_i, U_i)$, we remove the internal vertices of the $m-1$ paths $P_i$ (so that at most $7m$ vertices are removed). Removing only a few further vertices if necessary (but not removing $v_i$ or $u_i$), we may suppose that the resulting subsets $V_i' \subset V_i, U_i' \subset U_i$ satisfy 
\[ |V_i'| = |U_i'| = cn - 7m.\] 
Since $|V_i'| \geq \epsilon |V_i|$ and $|U_i'| \geq \epsilon |U_i|$, the fact that $(V_i, U_i)$ is $\epsilon$-regular implies that $d(V'_i, U'_i) \geq d' - \epsilon$. 
Moreover, by Lemma \ref{lem:epsregprop}(ii), as $cn \geq 14m$, the pair $(V_i', U_i')$ is $2\epsilon$-regular. Furthermore, since $\epsilon cn > 7m$, $v_i$ has at least 
$$(d'-\epsilon)cn - 7m > (d'-2\epsilon)(cn -7m)$$ neighbors in $U_i'$ and similarly for $u_i$. Let $\ell_0$ be the largest odd integer not larger than $\lfloor 2(1-2\sqrt{2\epsilon})(cn-7m) \rfloor$. 
By Lemma \ref{lem:countpath1} with $d$ replaced by $d' - \epsilon$, $\epsilon$ by $2\epsilon$, and $n$ by $cn-7m$,  for odd $l \leq \ell_0$, the number of paths of length $l$ connecting $v_i$ and $u_i$ is at least 
\begin{align} 
(d'&-\epsilon-7\sqrt{2\epsilon})^{l-1}   (2\epsilon (cn-7m))\prod_{i=1}^{l-2} (cn-7m-\lfloor i/2 \rfloor)
\geq  \ 2\epsilon (d'-8\sqrt{2\epsilon})^{l-1}\prod_{i=0}^{l-2} (cn-7m-\lfloor i/2 \rfloor) \nonumber
\\  \geq  \  & 2\epsilon (d'-8\sqrt{2\epsilon})^{l-1}\frac{2\pi}{e^2} ((cn-7m)/e)^{l-1} = \frac{4\epsilon \pi}{e^2} \left((d'-8\sqrt{2\epsilon})(cn-7m)/e\right)^{l-1} . \label{eq:pathsacross}
\end{align}
In the last inequality, we used the fact that  $a!/b!  \geq \frac{\sqrt{2\pi}}{e}\left( \frac{a}{e} \right)^{a-b} $ for positive integers $a > b$,  which easily follows from the upper and lower bounds in Stirling's approximation for factorials. Thus, $\frac{a!}{(a-l_1)!} \frac{a!}{(a-l_2)!} \geq \frac{2\pi}{e^2} \left( \frac{a}{e} \right)^{l_1+l_2}$, which we applied with $a=cn-7m$, $l_1=\lfloor (l-2)/2 \rfloor+1$, and $l_2=\lfloor (l-3)/2 \rfloor+1$. 

Therefore, within each bipartite graph $G(U_i', V_i')$, there are many choices for the path $T_i$ of any fixed odd length between $3$ and $\ell_0$.  Recall that the way we intend to build paths of length $k-1$ is by alternatingly concatenating $T_i$ and $P_i$. If all $m$ pairs give rise to a path $T_i$ of length $\ell_0$, the total length of these $T_i$'s, which is also a lower bound on the length of the path we build, is 
\begin{align} 
 m \ell_0 = & \ m (\lfloor 2(1-2\sqrt{2\epsilon})(cn-7m) \rfloor  - 1)  \nonumber \\
 \geq & \  m  (2(1-2\sqrt{2\epsilon})(cn-7m) -2 ) \nonumber 
\\
\geq & \   (2/3+\lambda)M \cdot (1-2\sqrt{2\epsilon})((1-\epsilon)n/M- 7m-2), \label{eq:atleastk}
\end{align}
where the last inequality holds because $m = (2/3+\lambda)M/2$, $2(1-2\sqrt{2\epsilon})>1$, and $c \geq (1-\epsilon)/M$. Since $\epsilon n /M > 7m+2$ and $\lambda = 13 \sqrt{\epsilon}$, (\ref{eq:atleastk}) is bounded below by  
\[
 (2/3+\lambda) \cdot (1-2\sqrt{2\epsilon})(1-2\epsilon)n
 > 2n/3 \geq k-1.    
\]

Hence, if all the paths $T_i$ are of length exactly $\ell_0$, the length of the full path we build would be larger than $k-1$.  Since the lengths of the $P_i$'s are fixed, while the length of $T_i$ can be any positive odd integer at least three and at most $\ell_0$, we will shorten some $T_i$ to make the path be of length exactly $k-1$. We greedily include $T_1, T_2, \dots$ such that each $T_i$ is of length $\ell_0$, until the length of the concatenated path $T_1, P_1, T_2, P_2 \dots $ is at least $k-1$ for the first time. If, when we stop, the total length is exactly $k-1$, we will take all those $T_i$ to have length $\ell_0$. Otherwise, when we stop, the total length is greater than $k-1$. If, when we stop, the last path is $T_j$ for some $j$, we will shorten the length of $T_j$ by deleting the last few vertices from $T_j$; 
if the last path is $P_j$ for some $j$, we will shorten the length of $P_j$ by deleting the last few vertices of $P_j$. 
In summary, there exists a properly chosen integer $m-1\geq m' \geq 0$ such that  $T_1, \dots, T_{m'}$ are of length $\ell_0$  
and, after  concatenating $T_1, P_1, T_2, P_2, \dots, T_{m'}, P_{m'}, T_{m'+1}$ with the length of $T_{m'+1}$ less than $\ell_0$ 
or $T_1, P_1, T_2, P_2, \dots, T_{m'}, P_{m'}$ with a possibly shortened $P_{m'}$, 
we obtain a path of length $k-1$. 
Using (\ref{eq:pathsacross}) to bound the number of $T_i$ for $i\leq m'$,  
 the total number of choices for $T_1, \dots, T_{m'}$ when fixing the anchor vertices $\{v_i, u_i\}$ for $1 \leq i\leq m'$ and $P_i$ for $1\leq i \leq m'$ is at least 
 \begin{align}\label{newhelpineq}
&  \left(\frac{4\epsilon\pi}{e^2}\right)^{m'} \cdot \left((d'-8\sqrt{2\epsilon})(cn-7m)/(2e)\right)
^{\ell_0 m'-m'}.
\end{align} 
If the concatenated path of length $k-1$ needs to end with a path $T_{m'+1}$ of length $1 \leq \ell' < \ell_0$, then $T_{m'+1}$ can be any path of length $\ell'$ alternating between $U_{m'+1}'$ and $V_{m'+1}'$ that starts with $v_{m'+1}$. 
By Lemma \ref{lem:countpath}  with $d$ replaced by $d' - \epsilon$, $\epsilon$ by $2\epsilon$, and $n$
 by $cn-7m$, since $1 \leq \ell' < \ell_0$, the number of choices for $T_{m'+1}$ is at least
  \begin{eqnarray*} 
  (d' -\epsilon-2\epsilon-\sqrt{2\epsilon})^{\ell'}\prod_{i=1}^{\ell'}((cn-7m)-\lceil i/2\rceil) & > & \frac{2\pi}{e^2} \left(  (d'-4\sqrt{2\epsilon})(cn-7m-1)/e  \right)^{\ell'} \\ & > & \frac{2\pi}{e^2} \left(  (d'-4\sqrt{2\epsilon})(cn-7m)/(2e)  \right)^{\ell'},
  \end{eqnarray*} 
  where the first inequality is by the same estimate as in (\ref{eq:pathsacross}). 
 Together with (\ref{newhelpineq}), 
  the total number of $k$-vertex paths when fixing the anchor vertices $\{v_i, u_i\}$ for $1 \leq i\leq m'$ and $P_i$ for $1\leq i \leq m'$ is at least 
 \begin{align}\label{newhelpineq2}
&  \left(\frac{4\epsilon\pi}{e^2}\right)^{m'+1} \cdot \left((d'-8\sqrt{2\epsilon})(cn-7m)/(2e)\right)
^{\ell_0 m'-m'+ \ell'}.
\end{align} 
Here we can assume $0 \leq \ell' < \ell_0$ to combine the two cases of whether the path of length $k-1$ ends with $T_{m'+1}$ or not. 
 
 Since the total length of $P_1,\dots, P_{m'}$ is at most $6m'$,
the total length of the $T_i$'s for $1 \leq i \leq m'+1$, which is $\ell_0m'+\ell'$, is at least $k -1 - 6m'$.
Thus, (\ref{newhelpineq2}) is at least 
\begin{align*}
 \left(\frac{4\epsilon\pi}{e^2}\right)^{m'+1}\cdot \left((d'-8\sqrt{2\epsilon})cn/(4e)
 \right)^{k-7(m'+1)}  
\geq  \epsilon^m \left((d'-8\sqrt{2\epsilon})cn/(4e)
\right)^{k -7m}. 
\end{align*}
Since $m = (2/3+\lambda)M/2$, which is a constant, there exists $c'$ such that the expression above is at least $(c' k)^k$, 
completing the proof.
\end{proof}

\begin{proof}[Theorem \ref{thm:main} for even cycles  for colorings satisfying Case 1 of Lemma \ref{lem:main}]
Fix $0< \epsilon \leq 10^{-20}$ and let $d=20\sqrt{\epsilon}$ and $\lambda = 13\sqrt{\epsilon}$.
The proof for even cycles is very similar to the previous proof for paths. Suppose there are vertex subsets $U_1,\ldots,U_m,V_1,\ldots,V_m$ with $m = (2/3 + \lambda) M/2$ and $|V_i|,|U_i|\geq cn$ satisfying the properties of Case 1 of Lemma \ref{lem:main}, say in color red. Let the edge density between $U_i$ and $V_i$ be at least $d'=d-\epsilon$. We may assume that $n$ is sufficiently large in terms of $c, \epsilon,$ and $M$. We will show that there is a constant $c'$ such that the number of monochromatic cycles with $k=\lceil 2(n+1)/3 \rceil$ vertices, with $k$ even, is at least $(c'k)^{k}$.

To do this, we will find distinct vertices $v_i^1, v_i^2, w_i^1, w_i^2 \in V_i$ for $2 \leq i \leq m-1$ and $v_1^1, w_1^1 \in V_1$ and $v_m^2, w_m^2 \in V_m$  such that (all the indices are mod $m$ and the edges considered are all in color red):
\begin{enumerate}
\item\label{cond:1} $v_i^1, v_i^2, w_i^1, w_i^2$ each have degree at least $(d'-\epsilon)|U_i|$ to $U_i$ for all $1 \leq i \leq m$;
\item \label{cond:2}there is a path $P_i$ connecting $v_i^1$ and $v_{i+1}^2$ and a path $Q_i$ connecting $w_i^1$ and $w_{i+1}^2$ such that both $P_i$ and $Q_i$ have length at most six; 
\item \label{cond:3}for each $i$, the lengths of $P_i$ and $Q_i$ have the same parity;
\item \label{cond:4} there is a path $L_i$ connecting $w_i^1$ and $w_i^2$ of length four;
\item \label{cond:5} all of the paths $P_i$, $Q_i$, and $L_i$ with $1 \leq i \leq m-1$ are vertex disjoint except where they share an end vertex.
\end{enumerate}

\begin{figure}[h]
\centering
\includegraphics[scale=0.8, trim={0cm 9.5cm 10cm 7.5cm},clip]{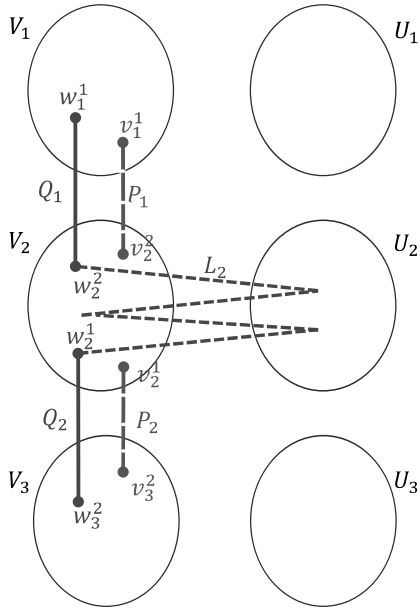}
\caption{An illustration showing the paths $P_i$, $Q_i$, and $L_i$.}
\end{figure}

Suppose that we can find such vertices $v_i^1, v_i^2, w_i^1, w_i^2$ together with appropriate paths $P_i$, $Q_i$, and $L_i$. We now show that we are done in this case. First we remove all the internal vertices in these paths from $U_i$ and  $V_i$; for $2 \leq i \leq m-1$, we also remove $w_i^1, w_i^2$ from $V_i$. This results in subsets $U_i'$ of $U_i$ and $V_i'$ of $V_i$. By the length constraints on $P_i$, $Q_i,$ and $L_i$ in conditions \ref{cond:2} and \ref{cond:4} above, all of these paths have in total at most 
\begin{equation}
5(m-1) + 5(m-1) + 3m  < 13m \label{eqn:internal}
\end{equation} 
 internal vertices. Thus, we have \[|V_i'| \geq |V_i| - 13m-2 \geq |V_i| - 15m \hspace{0.5cm} \textrm{and}  \hspace{0.5cm}   |U_i'| \geq |U_i| - 13m \geq|U_i| - 15m.\]  

Let $\ell_0$ be the largest even integer not larger than $\lfloor 2(1-2\sqrt{2\epsilon})(cn-15m)  \rfloor$.
Then, in the bipartite graph $G(U_1', V_1')$,  we will show that we can obtain many paths $T_1$ from $w_1^1$ to $v_1^1$ where $T_1$ has length $l_1 \leq \ell_0$. Clearly,  $T_1$ is of even length since it alternates between $U_1$ and $V_1$, eventually coming back to the side where it started. For $2 \leq i \leq m-1$, we show that we can find many paths $T_i$ from $v_i^2$ to $v_i^1$ in $G(V_i', U_i')$ where $T_i$ has even length $l_i \leq \ell_0$. In $G(U_m, V_m)$, we will find many paths $T_m$ from $v_m^2$ to $w_m^2$ where $T_m$ has even length $l_m \leq \ell_0$. Since  $v_i^1, v_i^2, w_i^1, w_i^2$ each have large degree to $U_i$, they also have large degree to $U_i'$. As in the proof of the path case, we can use Lemma \ref{lem:countpath1} applied to $G(V_i', U_i')$ to count the number of choices for the path $T_i$ given $l_i$.

As in the previous proof, we can create a cycle by concatenating 
\[T_1, P_1, T_2, P_2, \dots, T_{m-1}, P_{m-1}, T_{m}, Q_{m-1}, L_{m-1}, \dots, Q_2, L_2, Q_1.\] Since the $T_i$ and $L_i$ are all of even length and the lengths of $P_i$ and $Q_i$ have the same parity by Condition \ref{cond:3}, we obtain an even cycle. 

The total length of the even cycle we build is the total length of $P_i$, $Q_i$, and $L_i$ plus $\sum_{i=1}^m l_i$. 
If $l_i =4 $ for all $1 \leq i \leq m$, the total length is 
at most 
\[6(m-1) + 6(m-1)+ 4m + 4m<  20m < k.\]
On the other hand, if $l_i =\ell_0 $ for all $1 \leq i \leq m$, when $n$ is sufficiently large, the total length is 
at least 
\begin{align*} (m-1) + (m-1)+ 4m + \ell_0m & > (m-1) + (m-1)+ 4m +(\lfloor 2(1-2\sqrt{2\epsilon})(cn-15m)  \rfloor -1)m  \\
& > \lfloor 2(1-2\sqrt{2\epsilon})(cn-15m)  \rfloor   (2/3 + \lambda) M/2 \\
& \geq \lfloor 2(1-2\sqrt{2\epsilon})((1-\epsilon)n/M-15m)  \rfloor   (2/3 + \lambda) M/2 \\ 
& > (2/3+ \sqrt{\epsilon})n> k,
\end{align*}
where we used that $c \geq (1-\epsilon)n/M$ and $\lambda = 13 \sqrt{\epsilon}$. 
Therefore, we can reduce the lengths of some $l_i$, maintaining the condition that $4 \leq l_i \leq \ell_0$ are even integers for each $i$, to obtain a cycle of length exactly $k$.

 Thus, the total number of even cycles of length $k$, having fixed the $P_i$, $Q_i$, and $L_i$, is at least the product of the number of choices for $T_i$ for $1\leq i\leq m$. Since the total length of the $P_i$, $Q_i$, and $L_i$ is at most $6(m-1) + 6(m-1) + 4m < 16m$, the total length of the $T_i$, which is $\sum_{i=1}^m l_i$, is at least $k - 16m$. By a similar computation to (\ref{eq:pathsacross}) in the previous proof, the total number of even cycles of length $k$ is therefore at least
 \begin{align*} \prod_{i=1}^m 
  \left(\frac{4\epsilon\pi}{e^2}\right) (d'-8\sqrt{2\epsilon})^{l_i-1} ((cn-15m)/e)^{l_i-1}
& = \left(\frac{4\epsilon\pi}{e^2}\right)^m (d'-8\sqrt{2\epsilon})^{\sum_{i=1}^m l_i-m}((cn-15m)/e)^{\sum_{i=1}^m l_i-m} \\
& \geq \left(\frac{4\epsilon\pi}{e^2}\right)^m (d'-8\sqrt{2\epsilon})^{k-17m}((cn-15m)/e)^{k-17m},
 \end{align*}
 which is at least $(c'k)^k$ for some positive constant $c'$. 
It thus suffices to show that we can find vertices $v_i^1, v_i^2, w_i^1, w_i^2$ and paths $P_i$, $Q_i$, and $L_i$ satisfying Conditions \ref{cond:1} to \ref{cond:5}.

We will pick $v_i^1, v_i^2, w_i^1, w_i^2$ and $L_i, P_i,$ and $Q_i$ with the desired properties greedily. In step one, we pick four vertices $w_1^1, v_1^1, w_{2}^2, v_{2}^2$ and two paths $P_1$ and $Q_1$. 
In each step $i \geq 2$ except the last one, we pick four vertices $w_i^1, v_i^1, w_{i+1}^2, v_{i+1}^2$ and three paths $P_i, Q_i,$ and  $L_{i}.$ 
Suppose we have completed all steps $j<i$. We now need to pick $w_i^1, v_i^1 \in V_i$ and $w_{i+1}^2, v_{i+1}^2 \in V_{i+1}$.

Let distinct arbitrary vertices $v_1, v_1', v_1'' \in V_i, v_2, v_2', v_2'' \in V_{i+1}$ be such that $v_1, v_1', v_1''$ each have degree at least $(d'-\epsilon)|U_i|$ to $U_i$ and $v_2, v_2', v_2''$ each have degree at least $(d'-\epsilon)|U_{i+1}|$ to $U_{i+1}$. Since $(U_i, V_i)$ is $\epsilon$-regular, there are at least $(1-\epsilon)cn$ vertices in each of $V_i$ and $V_{i+1}$ that satisfy this degree condition from which $v_1, v_1', v_1''$ and $v_2, v_2', v_2''$ can be chosen. 
Since $\bigcup_{i=1}^m V_i \cup \bigcup_{i=1}^m U_i$ is $(200M, 6)$ well-connected in red, the pigeonhole principle implies that there are at least $100M$ red internally-disjoint paths connecting $v_1$ and $v_2$ whose lengths are at most six and of the same parity. Label $(v_1, v_2)$ as odd or even depending on the parity of the paths between them. We can similarly label $(v_1', v_2')$ and $(v_1'', v_2'')$. By the pigeonhole principle again, at least two of the pairs $(v_1, v_2)$, $(v_1', v_2')$, and $(v_1'', v_2'')$ have the same parity. Suppose $(v_1, v_2)$ and $(v_1', v_2')$ have the same parity, say odd. Then we let $v_1$ be $v_i^1$, $v_2$ be $v_{i+1}^2$, $v_1'$ be $w_i^1$, and $v_2'$ be $w_{i+1}^2$, noting that there are at least $100M$ internally vertex-disjoint paths connecting $v_1$ and $v_2$ of odd length at most 6 and the same for $v_1'$ and $v_2'$. Therefore, we have at least $100M$ candidates for $P_i$ and at least $100M$ candidates for $Q_i$. Since the previously chosen paths $P_j, Q_j,$ and $L_j$ use in total at most $7(m-1)+7(m-1)+ 5m< 19m$ vertices and $100M > 19m$, there are choices for $P_i$ and $Q_i$ with the desired properties.  

It remains to choose $L_i$. We remove all the internal vertices in the previously chosen $P_j$, $Q_j$, and $L_j$ from $U_i$ and  $V_i$ and we also remove $v_i^1$ and $v_i^2$ from $V_i$. This results in $U_i''$ and $V_i''$. By (\ref{eqn:internal}), \[|V_i''| \geq |V_i| - 13m-2 > |V_i| - 15m \]
and, similarly,
$|U_i''|  >|U_i| - 15m$. 
Furthermore, the pair $(U_i'', V_i'')$ is $2\epsilon$-regular with density $d(U_i'', V_i'') \geq d' - \epsilon$.

As $w_i^1$ and $w_i^2$ are in $V''_i$ and each has degree at least $(d'-2\epsilon)|U''_i|$ to $U''_i$, Lemma \ref{lem:countpath1} applied with $n$ replaced by $cn - 15m$, $\epsilon$ by $2\epsilon$, $d$ by $d'-\epsilon$, and $l$  by $4$ implies that there are at least 
$$(d'-\epsilon-7\sqrt{2\epsilon})^3 (2\epsilon) (cn-15m)^2 (cn-15m-1)\geq \epsilon (d-11\sqrt{\epsilon})^3 (cn)^3 /4$$
paths of length 4 connecting $w_i^1$ and $w_i^2$. Each vertex is in at most $3n^2$ paths with the prescribed end vertices, since $n^2$ is an upper bound on the number of choices for the other two internal vertices in this path and the multiplicative factor $3$ indicates which of the three internal vertices our vertex is.  Therefore, we have at least $\epsilon (d-11\sqrt{\epsilon})^3c^3 n/12$ vertex-disjoint paths of length $4$ connecting the two end vertices, each of which is a candidate for $L_i$, 
completing the proof.
\end{proof}

\subsubsection{Proof of Theorem \ref{thm:main} in the situation of Case 2 of Lemma \ref{lem:main}}\label{subsec:case2implypath}

We begin by showing that Theorem \ref{thm:main} is true for paths for edge-colorings satisfying Case 2 of Lemma \ref{lem:main}. The argument for even cycles will be almost the same. Throughout the proof, $n$ will be assumed to be sufficiently large in terms of $\epsilon$.

\begin{proof}[Theorem \ref{thm:main} for paths for colorings satisfying Case 2 of Lemma \ref{lem:main}]
Fix $0<\epsilon \leq 10^{-20}$ and let $d=20\sqrt{\epsilon}$ and $\lambda = 13\sqrt{\epsilon}$.  
We are given an extremal coloring with parameter $\alpha = 1000 (d+\lambda+\sqrt{\epsilon})
$, so we have a red/blue edge-coloring of a complete graph whose vertex set has a partition into subsets $V$ and $U$ with 
\begin{equation} \label{eqn:UV0}
|V| \geq (2/3 - \alpha)n, \ \ |U| \geq (1/3 - \alpha)n, \ \ |U|+ |V| = n. 
\end{equation}
Furthermore, without loss of generality, we can assume that the red density within $V$ is at least $1-\alpha$ and the blue density between $U$ and $V$ is at least $1-\alpha$. We want to prove that the number of monochromatic paths with $k$ vertices for $k = \lceil 2(n+1)/3 \rceil$ is at least $ \left(k/10\right)^{k}.$

We first perform a standard cleaning-up process, moving a few vertices between $U$ and $V$, so that within $V$ and between $U$ and $V$ certain degree conditions hold. 

\begin{claim}
[Updated Extremal Coloring]\label{claim:updatedEC}
There is a partition $V'\cup U'$  satisfying the following conditions:
\begin{itemize}
\item $|V'| \geq (2/3 - 3\alpha)n$, $|U'| \geq (1/3 - 2\alpha)n$.
\item The red graph on $V'$ has minimum degree at least $(2/3 - 4\alpha) |V'|$.
\item The blue density between $U'$ and $V'$ is at least $1-8\alpha$.
\item Each vertex in $U'$ has blue degree to $V'$ at least $(1/3 - 4\alpha)|V'|$. 
\end{itemize}
\end{claim}

\begin{proof}
We define $V'$ to be the set of vertices which have red degree at least $2|V|/3$ in $V$ and let $U'$ be the complement of $V'$, noting that each vertex in $U'$ has blue degree larger than $|V|/3-1$ to $V$.
We claim that this partition has the desired properties. 

We first show that most vertices of $V$ are in $V'$. 
Suppose  $|V\setminus V'| = x|V|$. Since the red density in $V$ is at least $1-\alpha$, we have 
$x \cdot 2|V|/3 + (1-x)|V| \geq (1-\alpha)|V|$ and, therefore, $x \leq 3\alpha$. Combining this inequality with (\ref{eqn:UV0}), we conclude that
\begin{equation*} |V'| \geq (1-3\alpha)|V| \geq (1-3\alpha) (2/3 - \alpha)n > (2/3 -3\alpha)n
. \label{eq:V'lb}
\end{equation*}

We next show that not many vertices in $U$ were moved to $V'$. 
Suppose  $|U\cap V'| = y|U|$. Since the red density between $U$ and $V$ is at most $\alpha$, we have 
$y \cdot 2|V|/3 \leq  \alpha|V|$ and so $y \leq 3\alpha/2$.
This implies that 
$|V| + 3\alpha |U|/2 \geq |V'|$. Using that $\alpha \leq 10^{-5}$ and (\ref{eqn:UV0}) gives 
\begin{equation}
|V'| \leq |V| + 3\alpha/2 \cdot |V| (1/3 + \alpha)/(2/3 - \alpha) < (1+ \alpha) |V|. \label{eq:V'Vratio}
\end{equation}
We also have
\begin{equation*} |V'|\leq |V| + 3\alpha |U|/2 \leq (2/3 + \alpha) n + 3\alpha/2 \cdot (1/3  + \alpha)n \leq (2/3 + 2\alpha) n \label{eq:V'ub}
\end{equation*}
and so $|U'|=n-|V'| \geq (1/3 - 2\alpha)n$. 

Furthermore, each vertex in $V'$ has red degree in $V'$ at least 
\[2|V|/3- 3\alpha|V|= (2/3 - 3\alpha) |V| \geq (2/3 - 3\alpha) |V'| / (1+\alpha) > (2/3 - 4\alpha) |V'|,
\]
where the second to last inequality is by (\ref{eq:V'Vratio}).

Similarly, together with (\ref{eq:V'Vratio}),
each vertex in $U'$ has blue degree to $V'$ at least \[|V|/3 -1 -3\alpha |V|  > \frac{1/3 - 3\alpha}{1+\alpha}|V'| - 1 > (1/3 - 4\alpha)|V'|.\]
The blue density between $U'$ and $V'$ is $e(U',V')/|U'||V'|$, which is at least 
\begin{align*} \frac{(1-\alpha)|V||U|  - |V\setminus V'||U| - \frac{1}{3}|U \setminus U'| |V|}{|U'||V'|}=\frac{(1-\alpha)|V||U|  - |V\setminus V'||U| - \frac{1}{3}|U \setminus U'| |V|}{(|U|+|V \setminus V'|-|U \setminus U'|)(|V|+|U \setminus U'|-|V \setminus V'|)}\\
\geq \frac{(1-\alpha)|V||U|  - |V\setminus V'||U| - \frac{1}{3}|U \setminus U'| |V|}{|V||U|+(|V|-|U|)|V \setminus V'|-(|V|-|U|)|U \setminus U'|}.
\end{align*}
We have already established that $|V \setminus V'| \leq 3\alpha |V|$ and $|U \setminus U'| = |U\cap V'| \leq 3\alpha|U|/2$. Hence, substituting for $|U \setminus U'|$ in the numerator of the last expression its maximum
$3\alpha|U|/2$ and in the denominator zero decreases the fraction.
Moreover, since $\alpha$ is sufficiently small, the last expression above is decreasing in $|V \setminus V'|$ and, therefore, minimized when $|V \setminus V'|=3\alpha |V|$. Hence, the blue density is at least 
\[\frac{(1-4.5\alpha)|V||U|}{|V||U| + (|V|-|U|)3\alpha|V|}> \frac{(1-4.5\alpha)|V||U|}{|V||U| + (\frac{2/3+\alpha}{1/3-\alpha}|U|-|U|)3\alpha|V|}
>  1-8\alpha,\] 
where we used that 
 $\alpha\leq 10^{-5}$ is sufficiently small and $|V|/|U| \leq (2/3+\alpha)/(1/3 - \alpha)$. 
\end{proof}

Abusing notation, we let $V'$ be the new $V$ and $U'$ the new $U$ and assume that they satisfy the properties described in Claim \ref{claim:updatedEC}. 
We now wish to count the number of monochromatic paths with $k$ vertices in this configuration. We have two cases. 

\paragraph{Case A:}
Suppose $|V| \geq k = \lceil 2(n+1)/3  \rceil$. Let $V'' \subset V$ be an arbitary subset with $|V''| = k$. The minimum red degree in $V''$ is at least
\[
(2/3 - 4\alpha)|V| - (|V| - \lceil 2(n+1)/3 \rceil) > (2/3 - 4\alpha)|V| - 4\alpha n  \geq  \frac{2}{3}k -8\alpha n \geq k/2.
\]
Therefore, the red graph on $V''$ is a Dirac graph. By the main result of \cite{CK}, the number of Hamiltonian cycles (and, hence, paths with $k$ vertices) in the red graph on $V''$ is at least $k!/2^{k+o(k)}$.

\paragraph{Case B:} Suppose $|V|< k$. In this case, $|U| \geq n - (k-1) \geq \lfloor k/2\rfloor$. To complete the proof, we apply Lemma \ref{manypathsdense} below to the blue bipartite graph with parts $U$ and $V$ with $\beta = 8\alpha $ and $\delta = (1/3-4\alpha)|V|$. Since $|V| \geq (2/3 - 3\alpha)n \geq 3k/4$ and every vertex in $U$ has blue degree to $V$ at least $(1/3-4\alpha)|V| \geq 4\sqrt{8\alpha}\max(|V|, 2|U|)$, the conditions of Lemma \ref{manypathsdense} are satisfied. Thus, the number of monochromatic blue paths with $k$ vertices is at least 
$$0.9^{|U|}2^{-k/2}0.94^{k/2}\lfloor k/2 \rfloor!(3k/4)!/(k/4)!>(k/10)^{k}.$$

\medskip

In either case, we get at least $\left(k/10\right)^{k}$ monochromatic paths with $k$ vertices, as required.  
\end{proof}

For a complete bipartite graph with parts $U$ and $V$, where $|V| \geq |U| \geq	 \lfloor k/2 \rfloor$, the number of paths with $k$ vertices starting in $V$ is precisely $(|V|)_{\lceil k/2 \rceil}(|U|)_{\lfloor k/2 \rfloor}$, where we use the standard falling factorial notation $(n)_k = n(n-1)\cdots(n-k+1)$. If the bipartite graph is not complete but just nearly complete, then, provided $|V|$ is much larger than $|U|$ and $U$ satisfies an appropriate minimum degree condition, we can prove that there are still almost this many paths with $k$ vertices between $U$ and $V$. This can be thought of as a counting version of a special case of the blow-up lemma \cite{KSS}. 

\begin{lemma}\label{manypathsdense}
Let $k$ be a sufficiently large positive integer and $G$ a bipartite graph with parts $U$ and $V$ such that $|V| \geq 3k/4$ and $|U| \geq \lfloor k/2 \rfloor$, the edge density between $U$ and $V$ is at least $1-\beta$ with $\beta<10^{-4}$, and every vertex in $U$ has degree at least $\delta \geq  4\sqrt{\beta}\max(|V|,2|U|)$. Then the number of paths with $k$ vertices in $G$ starting from a vertex in $V$ is at least $$\left(\frac{\delta}{4|V|}\right)^{2\sqrt{\beta}|U|}\left(1-\frac{4\sqrt{\beta}|U|}{\delta}\right)^{k/2}\left(1-6\sqrt{\beta}\right)^{k/2}(|U|)_{\lfloor k/2 \rfloor}(|V|)_{\lceil k/2 \rceil}.$$   
\end{lemma}

\begin{proof}
Let $U_0$ be the set of vertices in $U$ that have degree at most $(1-\sqrt{\beta})|V|$ and $U_1=U \setminus U_0$. The number of edges in $G$ satisfies 
$(1-\beta)|V||U| \leq e(G) \leq (1-\sqrt{\beta})|V||U_0|+|V|(|U|-|U_0|)$, from which we obtain $|U_0| \leq \sqrt{\beta}|U|$. 

We will show that there are many paths with $k$ vertices in $G$ alternating between $V$ and $U$ that start from a vertex in $V$. We do this by first showing that there are many sequences $L = u_1, \dots, u_{\lfloor k/2 \rfloor}$ of $\lfloor k/2 \rfloor$ distinct vertices in $U$ that extend to many paths with $k$ vertices in $G(U,V)$,  where extending here means that we can find vertices $v_1,\dots, v_{\lceil k/2\rceil}$ such that $v_1, u_1, v_2, u_2,\dots$ is a path with $k$ vertices, where the last vertex of the path is $u_{k/2}$ if $k$ is even and $v_{\lceil k/2 \rceil}$ if $k$ is odd.
 
We will require that the sequences $L$  satisfy the following property: 

\bigskip

\noindent (P) If $u_i \in L$ is in $U_0$, then $i \leq \delta /2$ and $i$ is odd. 

\bigskip

We first bound the number of choices for $L$. Note that if $i \leq \delta /2$ and even or $\lfloor k/2 \rfloor \geq i > \delta/2$, then $u_i$ must be in $U_1$. There are in total $\ell:=\lfloor k/2 \rfloor - \lceil \delta /4 \rceil$ such terms. Thus, we have $(|U_1|)_{\ell}$ choices for these terms in the sequence $L$. For the remaining $\lceil \delta /4 \rceil$ terms, we can choose any of 
the remaining vertices from $U$, so we get $(|U|-\ell)_{\lceil \delta /4 \rceil}$ possible choices to complete the sequence, giving a total of 

\begin{align*}\label{countL}
(|U_1|)_{\ell}(|U|-\ell)_{\lceil \delta /4 \rceil} & \geq  \left(\frac{|U_1|-\ell+1}{|U|-\ell+1}\right)^{\ell} (|U|)_{\ell}(|U|-\ell)_{\lceil \delta /4 \rceil} 
= \left(\frac{|U|-|U_0|-\ell+1}{|U|-\ell+1}\right)^{\ell} (|U|)_{\lfloor k/2 \rfloor} \\ 
& \geq \left(\frac{\lfloor k/2 \rfloor -|U_0|-\ell+1}{\lfloor k/2 \rfloor-\ell+1}\right)^{\ell} (|U|)_{\lfloor k/2 \rfloor}
\geq
 \left(\frac{\delta/4 - |U_0|}{\delta/4}\right)^{\ell} (|U|)_{\lfloor k/2 \rfloor}  \\ 
 & \geq  \left(1-4\sqrt{\beta}|U|/\delta\right)^{k/2}(|U|)_{\lfloor k/2 \rfloor}
 \end{align*}
possible sequences $L$. 

Having picked $L$, we greedily choose  $v_1,\ldots,v_{\lceil k/2 \rceil}$ to complete the path. We can pick $v_1$ to be any neighbor of $u_1$, so there are at least $\delta$ choices if $u_1 \in U_0$ and at least $(1-\sqrt{\beta})|V|$ choices if $u_1 \in U_1$. Having already picked out $v_1,\ldots,v_{j-1}$, we next show how to pick $v_j$. Note that (aside from the case\footnote{When $k$ is odd and $j=\lceil k/2\rceil$, we are instead picking a neighbor of $u_{(k-1)/2}$ not among $v_1,\ldots,v_{j-1}$, for which there are at least $(1-\sqrt{\beta})|V|-(j-1) > (1-3\sqrt{\beta})\left(|V|-(j-1)\right)$ choices, where we used $|V| \geq \frac{3}{4}k$ and $j = (k-1)/2$.} where $k$ is odd and $j=\lceil k/2\rceil$), this amounts to picking a common neighbor of $u_{j-1}$ and $u_j$ different from $v_1,\ldots,v_{j-1}$. Notice that, by property (P), no two consecutive terms of $L$ are in $U_0$. We thus have two cases to consider. 

In the first case, one of $u_{j-1}$ and $u_j$ is in $U_0$. In this case, we have $j-1 \leq \delta/2$ by property (P). The degree of the vertex from $U_0$ is at least $\delta $ and the degree of the other vertex, which is in $U_1$, is at least $(1-\sqrt{\beta})|V|$, so $u_{j-1}$ and $u_j$ have at least $\delta+  (1-\sqrt{\beta})|V|-|V|=\delta-\sqrt{\beta}|V|$ common neighbors. Hence, there are at least $\delta-\sqrt{\beta}|V|-(j-1) \geq \delta/2-\sqrt{\beta}|V| \geq \delta/4$ common neighbors of $u_{j-1}$ and $u_j$ not among $v_1,\ldots,v_{j-1}$. Any of these at least $\delta/4$ vertices can be chosen for $v_j$. 

In the second case, both $u_{j-1}$ and $u_j$ are in $U_1$. Then $u_{j-1}$ and $u_j$ have at least $2(1-\sqrt{\beta})|V|-|V|=(1-2\sqrt{\beta})|V|$ common neighbors, so there are at least $(1-2\sqrt{\beta})|V|-(j-1)>(1-6\sqrt{\beta})(|V|-(j-1))$ choices for $v_j$, where we used $|V| \geq \frac{3}{4}k$ and $j-1 < k/2$. 

As there are at most $|U_0|$ terms in the sequence $L$ that belong to $U_0$, there are at most $2|U_0|$ consecutive pairs in $L$ that include a term from $U_0$. Therefore, the first case happens at most $2|U_0|$ times.  Observe that $\delta/4 \leq |V|/4 \leq (1-6\sqrt{\beta})|V|/3 \leq (1-6\sqrt{\beta})(|V|-(j-1))$. From the estimates above, we therefore see that the number of ways of greedily choosing $v_1,\ldots,v_{\lceil k/2 \rceil}$ is at least
$$(\delta/4)^{2|U_0|}(1-6\sqrt{\beta})^{\lceil k/2 \rceil - 2|U_0|}(|V|)_{\lceil k/2 \rceil} / |V|^{2|U_0|} \geq (\delta/(4|V|))^{2\sqrt{\beta} |U|}(1-6\sqrt{\beta})^{k/2}(|V|)_{\lceil k/2 \rceil} .$$
Hence, by counting the number of choices for $L$ and then the number of ways of completing any given choice of $L$ to a path with $k$ vertices, we find that the number of paths with $k$ vertices in $G$ starting from a vertex in $V$ is at least the desired bound. 
\end{proof}

We now briefly discuss how to modify the argument above to prove Theorem \ref{thm:main} for even cycles for colorings satisfying Case 2 of Lemma \ref{lem:main}. 
Case A is identical, since we actually counted red cycles of length $k$ in the proof. Case B is almost identical, in that we only need to modify the proof of Lemma~\ref{manypathsdense} so that the conclusion guarantees many even cycles of length $k$ (instead of just paths with $k$ vertices) in the bipartite graph $G$. This amounts to also guaranteeing that $v_{1}$ is a neighbor of $u_{k/2}$ and only changes the bound slightly, so that, as in the case of paths, the number of blue cycles of length $k$ is at least $(k/10)^k$ for $k$ sufficiently large.

\section{Proof of Lemma \ref{lem:main}}\label{subsec:stability1}
Throughout this subsection $0< \epsilon \leq 10^{-10}$ and $d,\lambda \geq 1000\epsilon$, as in the assumptions of Lemma \ref{lem:main}. 

\subsection{Preparation}

In order to prove Lemma \ref{lem:main}, we first collect some auxiliary results, beginning with a lemma of Gy\'arf\'as, S\'ark\"{o}zy, and Szemer\'edi~\cite{stability} about finding a  well-connected subset in any red/blue-multicolored $K_n$.

\begin{lemma}[Lemma 4.1 in \cite{stability}]\label{lem:stab}
For every positive integer $t$ and red/blue-multicolored $K_n$, there exist $W \subset V(K_n)$ and a color, say red, such that $|W| \geq n-28t$ and $W$ is $(t,3)$-well-connected in the red subgraph of $K_n$.
\end{lemma}

We will also use two lemmas concerning extremal numbers of matchings. The first is a classical result of Erd\H{o}s and Gallai~\cite{EG}, which is easily seen to be tight by considering either the graph which consists of a clique on $2k+1$ vertices and a collection of isolated vertices or the graph in which the only edges are those incident to at least one of $k$ vertices. 

\begin{lemma}[Erd\H{o}s and Gallai \cite{EG}]\label{ErdosGallai} For integers $k$ and $n$ with $0 \leq k \leq n/2$, if the maximum matching in an $n$-vertex graph $G$ has size $k$, then $G$ has at most $\max\left({2k+1 \choose 2},{k \choose 2}+(n-k)k\right)$ edges. 
\end{lemma}

The second lemma about matchings that we will need is the following simple consequence of K\"onig's theorem, which says that the covering and matching numbers of a bipartite graph are equal. 
It is easily seen to be tight by considering the bipartite graph with $k$ vertices in one part complete to the other part, which has $n$ vertices, and no other edges.

\begin{lemma}\label{lem:Hall}
If a bipartite graph has at most $n$ vertices in each part and does not contain a matching of size larger than $k$, then it has at most $kn$ edges.
\end{lemma}
 
The final lemma in this section says that if a reduced graph satisfies certain properties, then the original graph it describes is well-connected.

\begin{lemma}\label{lem:F}
Let $H$ be a graph with vertex set $[h]$ in which each vertex has distance at most three from vertex $1$. Let $G$ be an $h$-partite graph with parts $V_1,\ldots,V_{h}$, each of order at least $N$. Suppose that $0< \alpha < 1/10$, $3\alpha<d<1$, and, for every edge $(i,j) \in E(H)$, the pair $(V_i,V_j)$ is $\alpha$-regular in $G$ with $d(V_i, V_j) \geq d$. Let $T \leq (d-3\alpha)N/5$ be a positive integer. Then, for each $i \in[h]$, there is $V_i' \subset V_i$ such that the following hold:
\begin{enumerate}
\item $|V_i'| \geq  (1-\alpha) |V_i|$.
\item For every edge $(i,j) \in E(H)$, the pair $(V_i',V_j')$ is $2\alpha$-regular with density $d(V_i', V_j') \geq d-\alpha$.
\item $\bigcup_{i=1}^{h} V_i'$ is $(T, 6)$-well-connected in $G$.
\end{enumerate}
\end{lemma}
\begin{proof}

For $i \in [h]$, let $s(i)$ denote the distance of $i$ from $1$. We are given $s(i) \leq 3$ for all $i \in [h]$. For $i \in [2,h]$, let $n(i)$ be an arbitrary neighbor of $i$ with $s(n(i))=s(i)-1$ and let $n(1)$ be an arbitrary neighbor of vertex $1$ in $H$. We call $n(i)$ the {\it successor} of vertex $i$. Let $D$ be the directed graph on $[h]$ in which each vertex $i$ has outdegree one with $n(i)$ as its outneighbor. 

For each $i \in [h]$, let $V_i'$ be the set of vertices in $V_i$ whose degree to $V_{n(i)}$ is at least $(d-\alpha)|V_{n(i)}|$. By  Lemma \ref{lem:epsregprop}(i),  $|V_i'| \geq (1-\alpha) |V_i|$. For each $(i,j) \in E(H)$, we have $d(V_i', V_j') \geq d-\alpha$ as $(V_i,V_j)$ is $\alpha$-regular and $1-\alpha > \alpha$. Moreover, by Lemma \ref{lem:epsregprop}(ii), since $\max(2\alpha, \alpha/(1-\alpha)) = 2\alpha$, $(V_i', V_j')$ is $2\alpha$-regular. It only remains to check Item 3 of the lemma, that is, to show that, for any two vertices $u, v \in \bigcup_{i=1}^{h} V_i'$, we can find $T$ internally-disjoint paths of length at most six connecting them. 

For each $i \in [h]$, there is a unique directed path $P_i$ in $D$ from $i$ to $1$. This path has length $s(i) \leq 3$ and the next vertex of the path is the successor of the current vertex. For each pair $(a,b)$ of not necessarily distinct vertices of $H$, let $W_{ab}$ be a walk in $H$ from $a$ to $b$ formed by concatenating a walk from $a$ to $1$ in $D$ of length two or three and a walk from $b$ to $1$ in $D$ of length two or three. Such a walk of length two or three from $a$ to $1$ in $D$ is either $P_a$ itself or formed  by adding to $P_a$ a walk of length two from $1$ to its successor and back. We can similarly  construct a walk of length two or three from $b$ to $1$ in $D$. 

Let $u \in V_a'$ and $v \in V_b'$, noting that $a$ and $b$ may not be distinct. Let $a=a_0,\ldots,a_s=b$ with $s \leq 6$ denote the vertices of the walk $W_{ab}$ from $a$ to $b$ in order and let $r$ be the length of the walk from $a$ to $1$ that makes up the first part of $W_{ab}$, so that $r=2$ or $3$ and $a_r=1$. In particular, $a_{j}$ is the successor of $a_{j-1}$ for $1 \leq j \leq r$ and $a_{j-1}$ is the successor of $a_j$ for $r < j \leq s$. 

We greedily construct $T$ internally vertex-disjoint paths from $u$ to $v$ of length at most $s$. Each such path has at most five internal vertices. In particular, after pulling out the internal vertices of $t<T$ such paths, all but at most $5t$ of the vertices in each part remain. The remaining subset $U_{i}$ of $V_{a_i}'$ for $1 \leq i \leq s-1$ has size $|U_{i}| \geq |V_{a_i}'|-5t  \geq (1-\alpha)|V_{a_i}|-5t$. We next build a walk $u=u_0,u_1,\ldots,u_s=v$ from $u$ to $v$ of length $s$ with $u_i \in U_{i}$ for each $i$. If this walk is a path, it is the desired next path from $u$ to $v$. Otherwise, we get the desired path by deleting some internal vertices from the walk. 

The vertex $u$ has at least $(d-\alpha)|V_{a_1}|$ neighbors in $V_{a_1}$, so $u$ has at least $(d-\alpha)|V_{a_1}|-(|V_{a_1}|-|U_{a_1}|) \geq (d-2\alpha)|V_{a_1}|-5t$ neighbors in $U_{a_1}$. These neighbors are all potential choices for $u_1$. As $(d-2\alpha)|V_{a_1}|-5t \geq \alpha |V_{a_1}|$, $d>\alpha$, the pair $(V_{a_1},V_{a_2})$ is $\alpha$-regular, and $U_{a_2} \subset V_{a_2}$, Lemma \ref{lem:epsregprop}(i) implies that all but at most $\alpha |V_{a_2}|$ vertices in $U_{a_2}$ have a common neighbor with $u$ in $U_1$ and thus can be chosen for $u_2$. If $r=3$, we similarly get that all but $\alpha |V_{a_3}|$ vertices in $U_{a_3}$ can be chosen for $u_3$. In either case, we get that all but at most $\alpha|V_{a_r}|$ vertices in $U_{r}$ can be chosen for $u_r$ when starting the walk from $u$. Similarly, working backwards from $v$, we get that all but at most $\alpha|V_{a_r}|$ vertices in $U_{r}$ can be chosen for $u_r$ when starting the walk from $v$. As $|U_{r}| \geq (1-\alpha)|V_{a_r}|-5t>2\alpha|V_{a_r}|$, there is a vertex in $U_r$ that can be chosen for $u_r$ to complete the walk from $u$ to $v$. Hence, we can continue the process of pulling out $T$ internally vertex-disjoint paths from $u$ to $v$, completing the proof. 
\end{proof}

\subsection{Proof of Lemma \ref{lem:main}}

Consider a red/blue edge-coloring of $K_n$.  Let $M_0$ be as in Szemer\'edi's regularity lemma, Lemma \ref{reglem}, with $\epsilon/2$ in place of $\epsilon$ and $m_0 = 1000/\epsilon$. 
Apply Lemma \ref{lem:stab} to this edge-coloring of $K_n$ with $t=200M_0$ to obtain a vertex subset $W$ with $|W| \geq n-5600M_0$ and a color such that $W$ is $(200M_0,3)$-well-connected in that color. By applying the  regularity lemma, Lemma \ref{reglem}, with $\epsilon/2$ in place of $\epsilon$ and $m_0 = 1000/\epsilon$ to the induced edge-coloring on $W$, we obtain the following lemma. 

\begin{lemma}\label{lem:firstW}
For every $0<\epsilon,d \leq 1/2$, there are positive integers $M_0$ and $n_0$ such that the following holds. For every red/blue edge-coloring of $K_n$ with $n \geq n_0$, there is a positive integer $1000/\epsilon \leq M \leq M_0$, a vertex subset $W \subset V(K_n)$ with $|W| \geq n-5600M_0$ which is $(200M_0, 3)$-well-connected in either the red or the blue subgraph, and an equitable partition $W=V_1 \cup \cdots \cup V_M$ such that the red/blue-multicolored reduced graph $H$ with vertex set $[M]$ and parameters $\epsilon/2$ and $d$ has at most $\epsilon {M \choose 2}/2$ non-adjacent pairs.
\end{lemma}

For the rest of this section, we fix an edge-coloring of $K_n$ with colors red and blue and the set $W$, the sets $V_1,\ldots,V_M$ in the equitable partition of $W$, and the reduced graph $H$  guaranteed by Lemma~\ref{lem:firstW}. We will also assume without loss of generality that $W$ is $(200M_0, 3)$-well-connected in red. Let $H_b$ be the spanning blue subgraph of $H$ and $H_r$ the spanning red subgraph, noting that the same edge can be in both $H_b$ and $H_r$. As $n$ is sufficiently large, for each $i \in M$, we have $|V_i| \geq \lfloor |W|/M \rfloor \geq \lfloor (n-5600M_0)/M \rfloor \geq (1-\epsilon/4)n/M$.

\begin{lemma}\label{lem:redreduced}
If the red subgraph $H_r$ contains a matching with at least $(2/3 + \lambda)M$ vertices, then the conditions of Case 1 in Lemma \ref{lem:main} are satisfied.
\end{lemma}

\begin{proof}
For each edge $(a_i, b_i)$ of such a matching in $H_r$, consider the corresponding sets of vertices $V_{a_i}$ and $V_{b_i}$. 
By construction, the union of these sets of vertices is a subset of a $(200M, 3)$-well-connected set in red. Furthermore, each of the parts has size at least $(1-\epsilon/4)n/M$ and each pair of parts corresponding to an edge of the matching in $H_r$ is $\epsilon$-regular of density at least $d$. Hence, all the conditions of Case 1 in Lemma \ref{lem:main} are indeed satisfied. 
\end{proof}

\begin{lemma}\label{lem:bluereduced}
If the blue subgraph $H_b$ contains a matching with at least $(2/3 + \lambda)M$ vertices and there is a vertex $v$ such that each vertex in the matching has distance at most three from $v$ in $H_b$, then the conditions of Case 1 in Lemma \ref{lem:main} are satisfied.
\end{lemma}

\begin{proof}
Let $S$ be the set of vertices with distance at most $3$ from $v$. Then, by assumption, $S$ contains the vertices of the blue matching. Furthermore, by Lemma \ref{lem:firstW}, for every $(i,j) \in E(H_b[S])$, $(V_i, V_j)$ is $\epsilon/2$-regular with blue density $d(V_i, V_j) \geq d$. Thus, we can apply Lemma \ref{lem:F} with $H$ being $H_b[S]$, $G$ being the blue $|S|$-partite graph induced on $\bigcup_{i\in S} V_i$, vertex $1$ being $v$, $\alpha=\epsilon/2$, $T=200M_0$, and $N=(1-\epsilon/4)n/M$.  As $n$ is sufficiently large, the conditions of  Lemma \ref{lem:F} are satisfied. Hence, for each $i \in S$, there is $V_i' \subset V_i$ such that $|V_i'| \geq (1-\epsilon/2)|V_i| \geq (1-3\epsilon/4)n/M$, for every edge $(i,j)$ of $H_b$ the pair $(V_i',V_j')$ is $\epsilon$-regular in $H_b$ with density $d(V_i', V_j') \geq d-\epsilon/2$, and $\bigcup_{i\in S} V_i'$ is $(200M, 6)$-well-connected. Since the vertices of the matching in $H_b$ are all in $S$, the conditions of Case 1 in Lemma \ref{lem:main} are satisfied. 
\end{proof}

For the rest of the section, we may therefore suppose that the largest matching in $H_r$ has  $2m <  (2/3 + \lambda)M$ vertices and no subgraph of $H_b$ with radius at most three contains a matching with $(2/3 + \lambda)M$ vertices. We will conclude that the given coloring of $K_n$ must be an extremal coloring with parameter $\beta:=1000(d+\lambda+\sqrt{\epsilon})$, which will complete the proof of Lemma~\ref{lem:main}. 

Consider a maximum matching in $H_r$ with $m$ edges $(a_i,b_i)$ for $1 \leq i \leq m$, so $m < (1/3 + \lambda/2)M$. Let $A = \{a_i:i \in [m]\}$, $B= \{b_i:i \in [m]\}$, and $C= [M] \setminus (A \cup B)$, so $A$, $B$, and $C$ form a partition of $[M]$ with $|A|=|B|=m$ and $|C|=M-2m>(1/3-\lambda)M$. We may assume without loss of generality that the red degree of $b_i$ to $C$ is at least the red degree of $a_i$ to $C$. Observe that $C$ contains no red edge as otherwise we could add it to the already constructed red matching, contradicting the fact that the chosen matching is maximum in $H_r$. Moreover, each $a_i$ has red degree to $C$ at most one, as otherwise there are red edges $(a_i,c_1)$ and $(b_i,c_2)$ with $c_1,c_2 \in C$ distinct and we could replace the edge $(a_i,b_i)$ in the matching by the two edges $(a_i,c_1)$ and $(b_i,c_2)$, making a larger matching in $H_r$ and again contradicting that the red matching is of maximum size. For the rest of the proof, we fix vertex subsets $A$, $B$, and $C$ with the properties described above. 

We prove several claims along the way to establishing that the coloring is an extremal coloring with parameter $\beta$. In outline, we will first show that the parts $A$, $B$, and $C$ each have roughly equal size by showing that $m$ is close to $M/3$ (this will follow from the upper bound on $m$ already given above and the lower bound on $m$ given in Claim~\ref{sclaim1} below). We will then deduce that we have an extremal coloring by showing that either the edges in $\bigcup_{i \in A \cup C} V_i$ are almost all blue and the edges from this set to $\bigcup_{i \in B} V_i$ are almost all red or the edges in $\bigcup_{i \in A \cup B} V_i$  are almost all red and the edges from this set to $\bigcup_{i \in C} V_i$ are almost all blue. Note that this will be sufficient as $W$ contains almost all vertices of $K_n$ and the regularity partition of $W$ is equitable. 

Recall that all edges of $H$ with both vertices in $C$ are blue. Let $v \in C$ be a vertex of largest blue degree in $C$ and $C' \subset C$ be the neighbors of $v$ in $C$ in the graph $H_b$. As there are at most $\epsilon{M \choose 2}/2$ non-adjacent pairs in $H$ and, hence, in the induced subgraph on $C$, by averaging, the vertex $v$ is in at most $\epsilon{M \choose 2}/|C| < 
\epsilon{M \choose 2}/\left((1/3-\lambda)M\right) < 2\epsilon M-1$ non-adjacent pairs. Hence, $|C'| \geq |C|-1-(2\epsilon M-1) = |C|-2\epsilon M$. 

Let $m_1$ be the size of a maximum blue matching between $A$ and $C'$. Pick a blue matching between $A$ and $C'$ of size $m_1$ whose vertices consist of subsets $A_1 \subset A$ and $C_1 \subset C'$ subject to the condition that $C' \setminus C_1$ contains a blue matching as large as possible. Let $C_2 \subset C' \setminus C_1$ be the vertices of this blue matching, so the vertices in $A_1 \cup C_1 \cup C_2$ are all in $C'$ or adjacent in blue to a vertex in $C'$ and therefore have distance in $H_b$ at most two from $v$ and are all in a blue matching with $2|A_1|+|C_2|$ vertices. By construction, $|A_1|=|C_1|=m_1$. Let $C_3 =C' \setminus (C_1 \cup C_2)$, so $C'=C_1 \cup C_2 \cup C_3$ forms a partition of $C'$ into three parts. Figure \ref{fig1} illustrates the different sets.

\begin{figure}[h]
    \centering
    \includegraphics[trim=0 100 0 100,clip, scale=0.45]{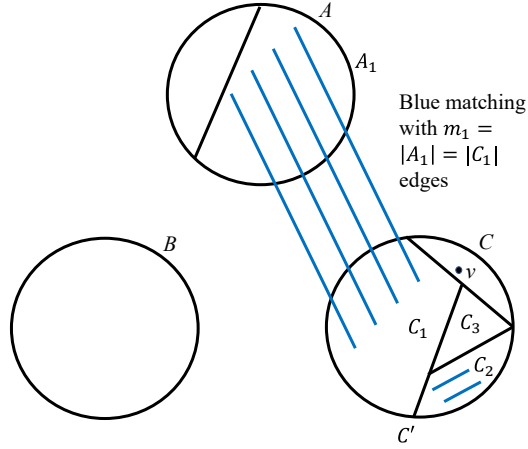}
    \caption{An illustration showing $A$, $B$, $C$, $A_1$, $C_1$, $C_2$, and $C_3$. $C'$ is the set of blue neighbors of the vertex $v$ and $C_1$, $C_2$, and $C_3$ form a partition of $C'$.}
    \label{fig1}
\end{figure}

By the choice of the maximum blue matching, there are no blue edges from $A \setminus A_1$ to $C' \setminus C_1$. Moreover, for each edge $(a,c)$ of the maximum blue matching between $A$ and $C'$ (so $a \in A_1$ and $c \in C_1$), $a$ has no blue edges to $C' \setminus C_1$ or $c$ has no blue edges to $A \setminus A_1$.  On the other hand, each vertex in $A$ has at most one red edge to $C$, $C$ contains no red edges, and there are at most $\frac{\epsilon}{2}{M \choose 2}$ non-adjacent pairs. Comparing these upper and lower bounds on the number of non-blue pairs in $A \cup C'$ with not both vertices in $A$, we obtain 
\begin{equation}\label{AACC} |A \setminus A_1||C' \setminus C_1|+|A_1|\min(|A \setminus A_1|,|C' \setminus C_1|) \leq 
 \frac{\epsilon}{2}{M \choose 2}+|A| \leq \frac{2}{7}\epsilon M^2,\end{equation}
where the last inequality is from $|A| \leq M$ and $M \geq m_0 = 1000/\epsilon$. 

For each edge $(a,c) \in A_1 \times C_1$ in the maximum blue matching between $A$ and $C'$, either $a$ or $c$ has blue degree to $C_3$ at most one, since otherwise we can replace $(a,c)$ by two blue edges $(a,c_1)$ and $(c,c_2)$ with $c_1,c_2 \in C_3$, which would also give a maximum blue matching between $A$ and $C'$, but would increase the size of the maximum blue matching in the remaining vertices in $C'$, contradicting our choice of the blue matching between $A$ and $C'$. Hence, there are at least $|A_1|(|C_3|-1)$ pairs between $A_1 \cup C_1$ and $C_3$ which are not blue. Moreover, for each matching edge $(c_3,c_4)$ in the maximum blue matching in $C' \setminus C_1$ (so $c_3,c_4 \in C_2$), either $c_3$ or $c_4$ has blue degree at most one to $C_3$, so there are at least $(|C_2|/2)(|C_3|-1)$ pairs between $C_2$ and $C_3$ which are not blue. Finally, there are no blue edges in $C_3$. Hence, comparing the upper and lower bounds on the number of non-blue pairs between $A_1 \cup C'$ and $C_3$, we similarly obtain 
\begin{equation}\label{AACC1} (|A_1|+|C_2|/2+|C_3|/2)(|C_3|-1)
 \leq 
 \frac{\epsilon}{2}{M \choose 2}+|A_1| \leq \frac{2}{7}\epsilon M^2.\end{equation}

\begin{sclaim}\label{sclaim1}
$m \geq \left(\frac{1}{3}-\lambda-4\epsilon\right)M$. 
\end{sclaim}

\begin{proof}
Suppose, for the sake of contradiction, that $m <  \left(\frac{1}{3}-\lambda-4\epsilon\right)M$. Then $|C'| \geq |C|-2\epsilon M = M-2m-2\epsilon M \geq m=|A|$. In particular, $\min(|A \setminus A_1|,|C' \setminus C_1|)=|A \setminus A_1|$. Since also $|A_1|=|C_1|=m_1$, the left-hand side of (\ref{AACC}) simplifies to $|C'|(m-m_1)$. Hence, (\ref{AACC}) implies that $|C'|(m-m_1) \leq \epsilon M^2/3$. As $|C'| \geq M-2m-2\epsilon M \geq M/3$, we obtain $m_1 \geq m-\epsilon M$.

As $|A_1|=|C_1| = m_1$ and $|C'|=|C_1|+|C_2|+|C_3|$, we have $$|A_1|+|C_2|/2+|C_3|/2=(|C'|+m_1)/2 \geq (M-m-3\epsilon M)/2 \geq M/3.$$ Hence, from (\ref{AACC1}), we similarly obtain $|C_3| \leq \epsilon M$. 

Thus, the number of vertices of the blue matching of distance at most two from $v$ is \begin{eqnarray*}
2|A_1|+|C_2| & = &  2m_1+|C'|-m_1-|C_3| \geq m_1+M-2m-2\epsilon M-\epsilon M  \\ & \geq & M-m-4\epsilon M  > (2/3+\lambda)M,
\end{eqnarray*}
contradicting the assumption that no such large blue matching exists. 
\end{proof}

If $|C' \setminus C_1| \geq |A \setminus A_1|$, the left-hand side of (\ref{AACC}) is equal to  $|C'|(m-m_1)$. Otherwise, $|A \setminus A_1| > |C' \setminus C_1|$ and the left-hand side of (\ref{AACC}) is equal to $m(|C'|-m_1)$. In either case, as $m,|C'| \geq \frac{2}{7} M$, we obtain from (\ref{AACC}) that 
\begin{equation}
\label{m1m} |A_1|=m_1 \geq \min(m,|C'|)-\epsilon M.
\end{equation}

Consider a maximum blue matching between $A_1$ and $B$. Let $m_2$ be the number of edges of this blue matching and let $A_2 \subset A_1$ and $B_2 \subset B$ be the vertices in this blue matching. Consider a blue matching between $A_1 \setminus A_2$ and $C_1$ that matches every vertex in $A_1 \setminus A_2$ subject to the condition that the vertices in $C'$ not contained in this blue matching  contain a blue matching of maximum possible size. Note that such a blue matching between $A_1 \setminus A_2$ and $C_1$ exists as there is a perfect matching between $A_1$ and $C_1$ by construction. Let $C_1'$ be the set of $|A_1 \setminus A_2|$ vertices in $C_1$ that match with a vertex in $A_1 \setminus A_2$ and $C_4 \subset C' \setminus C_1'$ consist of the vertices in the maximum blue matching in $C' \setminus C_1'$. Let $C_5=C' \setminus (C_1' \cup C_4)$. Figure \ref{fig2} is an illustration of these sets.
\begin{figure}[h]
    \centering
    \includegraphics[trim=0 100 0 100,clip, scale=0.45]{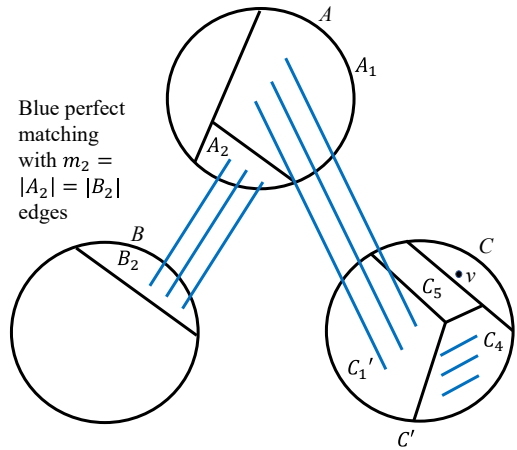}
    \caption{An illustration showing $A_2 \subset A_1$, $B_2$, $C_1' \subset C_1$, $C_4$, and $C_5$. The sets of blue edges represent three distinct blue matchings.}
    \label{fig2}
\end{figure}

There are clearly no blue edges in $C_5$. As before, for each edge $(a,c)$ in the blue matching between $A_1 \setminus A_2$ and $C_1'$, either $a$ or $c$ has blue degree at most one to $C_5$. Similarly, for each edge $(c_1,c_2)$ of the blue perfect matching in $C_4$, either $c_1$ or $c_2$ has blue degree at most one to $C_5$. Hence, the number of pairs between $(A_1 \setminus A_2) \cup C'$ and $C_5$ which are not blue is at least $(|C_1'|+|C_4|/2+|C_5|/2)(|C_5|-1)$ and at most $\frac{\epsilon}{2}{M \choose 2}+|A_1 \setminus A_2| \leq \epsilon M^2/4 + M$ (recall that $C$ has no red edges and 
each vertex in $A$ has at most one red neighbor in $C$). As 
$$|C_1'|+|C_4|/2+|C_5|/2 = |C'|/2+|C_1'|/2 \geq |C'|/2 \geq M/7,$$ 
we obtain that $|C_5| \leq \frac{7}{4}\epsilon M +8  \leq 2\epsilon M$. 

We have obtained a blue matching with vertex set $B_2 \cup A_1 \cup C_1' \cup C_4$. Each vertex in this blue matching has distance in $H_b$ at most three from $v$ and, therefore, the number of vertices in this blue matching is less than $(2/3+\lambda)M$. On the other hand, the number of vertices in this blue matching is at least 
\begin{eqnarray*}
|B_2|+m_1+|C'|-|C_5| & \geq & |B_2|+\min(m,|C'|)+|C'|-3\epsilon M \geq |B_2|+\min(m,|C|)+|C|-7\epsilon M \\ & >
 & |B_2|+\left(\frac{2}{3}-2\lambda-7\epsilon\right)M,
\end{eqnarray*}
where the first inequality uses (\ref{m1m}) and the last inequality uses $2m < (2/3+\lambda)M$ and $|C| +m=M-m$ when $m \leq |C|$ and uses $|C| > (1/3-\lambda)M$ when $m > |C|$. 
We thus have the following claim. 

\begin{sclaim}\label{sclaim2}
$|B_2| \leq (3\lambda+7\epsilon) M$. 
\end{sclaim}
 
In particular, as there are no blue edges between $B \setminus B_2$ and $A_1 \setminus A_2$, the graph between $A$ and $B$ is almost entirely red. Let $\mu:=(6\lambda+16\epsilon)M$. 

\begin{sclaim}\label{sclaim3}
The largest red matching with vertices in $A$ has size less than $\mu$ or the largest red matching between $B$ and $C'$ has size less than $2\mu$. 
\end{sclaim}

\begin{proof}
Suppose, for the sake of contradiction, that the claim does not hold. Consider a red matching in $A$ of size $\mu$ and let $A_3$ be the set of vertices of this red matching, so $|A_3|=2\mu$. Consider a red matching between $B$ and $C'$ of size $2\mu$ and let $B_3$ be the vertices of this matching in $B$. Figure \ref{fig3} illustrates these sets. 

\begin{figure}[h]
    \centering
    \includegraphics[trim=0 100 0 100,clip, scale=0.45]{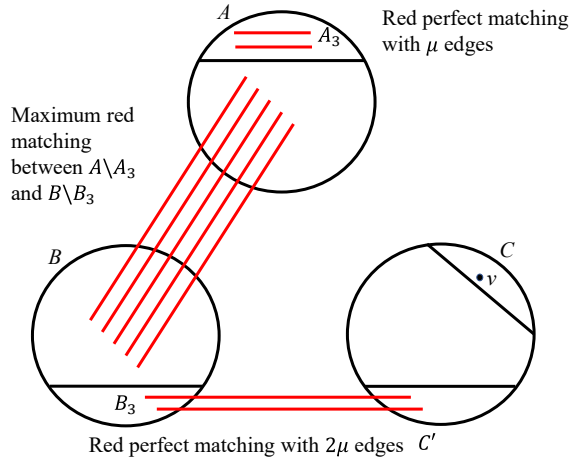}
    \caption{An illustration showing $A_3$ and $B_3$.  }
    \label{fig3}
\end{figure}

Observe that $|A \setminus A_3|=|B \setminus B_3|=m-2\mu$. Edge partition the balanced complete bipartite graph between $A \setminus A_3$ and $B \setminus B_3$ into $m-2\mu$ perfect matchings, so each of the perfect matchings has exactly $m-2\mu$ edges. The number of missing edges in $H$ is at most $\frac{\epsilon}{2}{M \choose 2}<\epsilon M^2/4$ and $m-2\mu \geq M/4$, so the density of non-adjacent pairs between  $A \setminus A_3$ and $B \setminus B_3$ in $H$ is at most $4\epsilon$ and there is a matching $\mathcal{M}$ in $H$ between $A \setminus A_3$ and $B \setminus B_3$ with at least $(1-4\epsilon)(m-2\mu)$ edges by Lemma \ref{lem:Hall}. By Claim \ref{sclaim2}, the maximum size of a blue matching between $A_1$ and $B$ is at most  $(3\lambda+7\epsilon ) M$, so the maximum size of a blue matching between $A$ and $B$ is at most $$(3\lambda +7\epsilon) M+|A|-|A_1| \leq (4.5\lambda+10\epsilon) M,$$ 
where the last inequality uses (\ref{m1m}) (so that $|A|-|A_1| = m - m_1  \leq m - \min(m, |C'|)+\epsilon M$),  
$|C'| \geq |C|-2\epsilon M=M-2m-2\epsilon M$, and $m<(1/3+\lambda/2)M$.
Hence, the matching $\mathcal{M}$ has at least $(1-4\epsilon)(m-2\mu)- (4.5\lambda+10\epsilon)M$ red edges. This red matching, together with the red matching of size $\mu$ with vertex set $A_3$ and the red matching between $B_3$ and $C_3$ of size $2\mu$, forms a red matching of size at least \begin{eqnarray*}(1-4\epsilon)(m-2\mu)- (4.5\lambda+10\epsilon)M+3\mu & \geq & m-(4.5\lambda+12\epsilon)M +\mu \\ & \geq & \left(1/3 - 5.5\lambda-16\epsilon \right)M +\mu \\ & \geq &   (2/3+\lambda)M/2,\end{eqnarray*} 
where the first inequality uses $m \leq (2/3+\lambda)M/2 \leq M/2$, the second inequality follows from Claim \ref{sclaim1}, and the final inequality from the definition of $\mu$. This contradicts the assumption that $H_r$ has no red matching of size $(2/3+\lambda)M/2$. 
\end{proof}

From Claim \ref{sclaim3}, the rest of the proof naturally splits into two cases.

\paragraph{Case 1:} The largest red matching with vertices in $A$ has size less than $\mu$. 
\vspace{3mm}

Let $\tau:=2\mu+2\lambda M +3\epsilon M+1$. Suppose, for the sake of contradiction, that there is a blue matching between $C'$ and $B$ of size $\tau$. Let $C_6 \subset C'$ be the $\tau$ vertices of $C'$ in this blue matching. Figure \ref{fig4} is an illustration. 

\begin{figure}[h]
    \centering
    \includegraphics[trim=0 100 0 100,clip, scale=0.45]{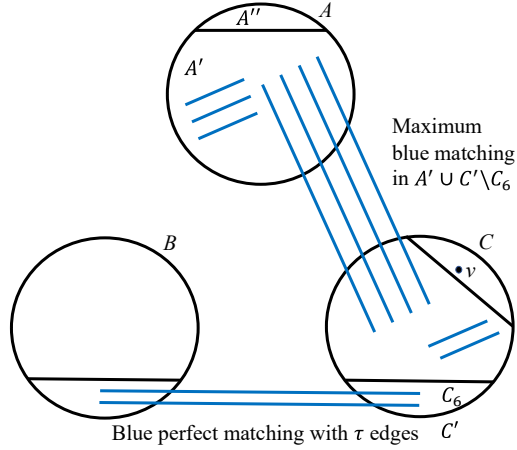}
    \caption{An illustration showing $A'$, $A''$, and $C_6$.  }
    \label{fig4}
\end{figure}

Let $A' \subset A$ be those vertices with at least one blue neighbor in $C'$ and $A''=A \setminus A'$. Each vertex in $A$ has red degree to $C'$ at most one and so each vertex in $A''$ is in at most one edge to $C'$. As there are at most $\epsilon M^2/4$ non-adjacent pairs in total and $|C'| \geq M/4 + 1$, we thus have $|A''|(|C'|-1) \leq \epsilon M^2/4$ and hence $|A''| \leq \epsilon M$. 

We next bound the number of pairs of vertices in $A' \cup C' \setminus C_6$ which are not blue edges. There are at most $\frac{\epsilon}{2}{M \choose 2}$ non-adjacent pairs. There are also no red edges in $C'$.  By Lemma \ref{ErdosGallai}, there are fewer than ${\mu \choose 2}+\mu(|A'|-\mu)$ red edges in $A'$. Finally, as each vertex in $A$ has red degree at most one to $C$, there are at most $|A'|$ red edges between $A'$ and $C'$. In total, the number of pairs of vertices in $A' \cup C' \setminus C_6$ which are not blue is at most 
\begin{eqnarray*}\frac{\epsilon}{2}{M \choose 2}+{\mu \choose 2}+\mu(|A'|-\mu)+|A'|  & \leq & \mu |A'|-\mu^2/2 +|A'|+\epsilon M^2/4 
\\ & \leq & 2\mu |A'|
\\ & \leq & 2\mu (|A'|+|C'|-\tau - \mu -1/2) 
 \\ & =& 2\mu \left(|A'|+|C' \setminus C_6|-\mu-1/2\right) \\ & = & {|A'|+|C' \setminus C_6| \choose 2}- {|A'|+|C' \setminus C_6|-2\mu \choose 2}.\end{eqnarray*}
Here we used ${x \choose 2}-{x-y \choose 2}=y\left(x-y/2-1/2\right)$ with $x=|A'|+|C' \setminus C_6|$ and $y=2\mu$. 
Therefore, there are at least ${|A'|+|C' \setminus C_6|-2\mu \choose 2}$ blue edges with both vertices in $A' \cup C' \setminus C_6$. 
By Lemma \ref{ErdosGallai}, there is a blue matching in $A' \cup C' \setminus C_6$ spanning at least $|A'|+|C' \setminus C_6|-2\mu-1$ vertices. The blue matching consisting of the blue matching between $C_6$ and $B$ of size $\tau$ together with the maximum blue matching in $A' \cup C' \setminus C_6$ contains only vertices of distance at most two from $v$ and has at least \begin{eqnarray*} 2\tau+|A'|+|C' \setminus C_6|-2\mu -1 & = & \tau+|A'|+|C'|-2\mu -1 \geq \tau+|A|+|C|-3\epsilon M-2\mu -1 \\ & = & \tau+M-m-3\epsilon M-2\mu -1  = 2\lambda M +M-m \geq (2/3+\lambda)M\end{eqnarray*} vertices, contradicting our assumption that there is no such blue matching. Hence, there is no blue matching between $C'$ and $B$ of size $\tau$. 

We now show that the coloring is an extremal coloring with parameter $\beta = 1000(d+\lambda+\sqrt{\epsilon})$ with the set $\bigcup_{i \in A \cup C} V_i$ almost entirely blue and the bipartite graph between $\bigcup_{i \in A \cup C} V_i$ and $\bigcup_{i \in B} V_i$ almost entirely red. 
We first check that the two parts of this partition are of the claimed size. It suffices to check that $\left|\bigcup_{i\in B} V_i \right|/n$ is $\frac{1}{3}\pm \beta$. Since $W \subset V(K_n)$ satisfies $|W|/n \geq 1-\epsilon/4$ and the partition $W=\bigcup_{i\in [M]} V_i$ is equitable with $n \gg M$, it is enough to show that $|B|/M$ is $\frac{1}{3}\pm \frac{\beta}{2}$. But this follows easily, since $|B|=m$ and $\frac{1}{3}-\lambda-4\epsilon < m/M \leq \frac{1}{3}+\lambda/2$, where the first inequality is by Claim \ref{sclaim1}.

We next show that the induced subgraph on $\bigcup_{i \in A \cup C} V_i$ is almost entirely blue. The number of non-edges in $A \cup C$ is at most $\frac{\epsilon}{2}{M \choose 2}$. There are also no red edges in $C$. The number of red edges from $A$ to $C$ is at most $|A|$. By Lemma \ref{ErdosGallai}, the number of red edges in $A$ is at most ${\mu \choose 2}+\mu(|A|-\mu)$. Each part $V_i$ in the equitable partition has size at most $\lceil n/M\rceil$ and the density between any $\epsilon/2$-regular pair of parts that does not correspond to a red edge in $H$ is at most $d$. Hence, the total number of edges with both vertices in $\bigcup_{i \in A \cup C} V_i$ which are red is at most  
$$\left(\frac{\epsilon}{2}{M \choose 2}+|A|+{\mu \choose 2}+\mu(|A|-\mu)\right)\lceil n/M\rceil^2+d{n \choose 2} \leq \left(d+\frac{\epsilon}{2}+\frac{\mu}{M}\right){n \choose 2} \leq \beta {\left|\bigcup_{i \in A \cup C} V_i\right| \choose 2},$$
where we used that $\left|\bigcup_{i \in A \cup C} V_i\right| \geq 3n/5$ and $\beta \geq 4d+2\epsilon+4\mu/M$.

Finally, we show that the bipartite graph between $\bigcup_{i \in A \cup C} V_i$ and $\bigcup_{i \in B} V_i$ is almost entirely red. 
We first bound the number of edges between $A \cup C$ and $B$ which are not red. 
The number of missing edges between $A \cup C$ and $B$ is at most $\frac{\epsilon}{2}{M \choose 2}$. 
Recall that the maximum blue matching from $A_1$ to $B$ is of size $|B_2|$. Moreover, the maximum blue matching between $C'$ and $B$ is of size less than $\tau$ and hence, by Lemma \ref{lem:Hall}, the number of blue edges between $A \cup C$ and $B$ is less than \begin{align*} 
&\left(|B_2|+|A \setminus A_1|+|C \setminus C'|\right)|B| + \max(|C’|,|B|)\tau \\
  \leq & \left(\left(3\lambda M+7\epsilon M\right)+\left(m-\min(m,|C'|)+\epsilon M\right)+2\epsilon M \right)m + \tau M \\
   = & \left(3\lambda+10\epsilon \right)Mm+(m-\min(m,|C'|))m + \tau M \\
   \leq & \left(3\lambda+12\epsilon \right)Mm+(m-\min(m,|C|))m + \tau M \\
      \leq & \left(4.5\lambda+12\epsilon \right)Mm+ \tau M
       \\
     < & (20\lambda + 50\epsilon) M^2 ,
  \end{align*}
where we used Claim \ref{sclaim2} and (\ref{m1m}) in the first inequality, $|C'| \geq |C|-2\epsilon M$ in the second inequality, $m-\min(m,|C|) \leq 1.5\lambda M$ in the third inequality, and, in the last inequality, we substituted in the values of $\tau$ and $\mu$ and used the lower bound on $M$. 
Hence, the number of blue edges between $\bigcup_{i \in A \cup C} V_i$ and $\bigcup_{i \in B} V_i$ is at most 
$$\left(\frac{\epsilon}{2}{M \choose 2}+(20\lambda + 50\epsilon) M^2\right)
\lceil n/M\rceil^2+d{n \choose 2} \leq \big(d+20\lambda+51\epsilon \big)n^2 \leq \beta \left|\bigcup_{i \in A \cup C} V_i\right| \cdot \left|\bigcup_{i \in B} V_i\right|,$$
where we used that $3n/5 \leq \left|\bigcup_{i \in A \cup C} V_i\right| \leq 7n/10$ and also that $\beta \geq 6(d+20\lambda+ 51\epsilon)$. This completes the proof in this case.

\paragraph{Case 2:} The largest red matching between $B$ and $C'$ has size less than $2\mu$. 

\vspace{3mm}

Our goal is to show that the coloring is an extremal coloring with parameter $\beta = 1000(d+\lambda+\sqrt{\epsilon})$ with at most a $\beta$-fraction of the edges in $\bigcup_{i \in A\cup B} V_i$ blue and at most a $\beta$-fraction of the edges between $\bigcup_{i \in A\cup B} V_i$ and $\bigcup_{i \in C} V_i$ red. To show these two parts have the desired size, it suffices to show that $\bigcup_{i \in C} V_i$ has size $(1/3\pm \beta)n$ and this follows from a very similar computation to that in Case 1.  

Let $B_4$ be the set of vertices in $B$ that have at least one blue neighbor in $C'$, so every vertex in $B_4$ has distance at most two from $v$ in blue.  Since there is no red matching between $B \setminus B_4$ and $C'$ of size $2\mu$, Lemma \ref{lem:Hall} implies that the number of red edges between $B\setminus B_4$ and $C'$ is at most $2\mu \max(|C'|, |B\setminus B_4|)$. However, since the edges between $B \setminus B_4$ and $C'$ are all red and there are at most $\frac{\epsilon}{2}{M \choose 2}$ non-edges, the number of red edges between $B\setminus B_4$ and $C'$ is at least $|B \setminus B_4||C'| - \frac{\epsilon}{2}\binom{M}{2}$. Since  $|C'| \geq |C| - 2\epsilon M \geq (1/3-\lambda-2\epsilon)M$, it follows that $|B \setminus B_4| \leq 3\mu$. 
We get a blue matching with vertices of distance at most two from $v$ by taking a maximum blue matching between $A$ and $C'$ (which is of size $m_1$) together with a maximum blue matching in $B_4$, which is of size $m_4$, say. Together this matching has size $m_1+m_4$ and so, by assumption, we have $m_1+m_4 < (2/3+\lambda)M/2$. Hence, together with (\ref{m1m}), $$m_4 <  (2/3+\lambda)M/2 -m_1 \leq (2/3+\lambda)M/2 -\min(m,|C'|)+\epsilon M \leq \frac{3}{2}\lambda M + 5\epsilon M,$$
where we used that $m \geq \left(\frac{1}{3}-\lambda-4\epsilon\right)M$ by Claim~\ref{sclaim1} and $|C'| \geq |C| - 2\epsilon M > \left(\frac{1}{3}-\lambda-2\epsilon\right)M$. 
By Lemma \ref{ErdosGallai}, with $k=\frac{3}{2}\lambda M + 5\epsilon M$, the number of blue edges in $B_4$ is thus at most 
$|B_4|k \leq km$. The number of pairs of vertices in $B$ that are not in $B_4$ is at most  $|B \setminus B_4 | |B| \leq 3\mu m$. Hence, there are at most $\left(\frac{3}{2}\lambda M + 5\epsilon M+3\mu\right)m \leq 4\mu m$ blue edges in $B$. 

We next bound the number of blue edges in $A$. We first claim that there is a matching (which does not have to be monochromatic) in $B_4 \cup C'$ with each edge containing at most one vertex in $B_4$ and with at least $|C'|+\min(|B_4|,|C'|)-5\sqrt{\epsilon} M$ vertices. Indeed, if $|B_4| \geq |C'|$, then, since the number of edges between $B_4$ and $C'$ is at least $|B_4||C'| - \frac{\epsilon}{2}\binom{M}{2}$, Lemma \ref{lem:Hall} gives a matching of size $|C'| - \epsilon M$, better than desired. On the other hand, if $|B_4| < |C'|$, Lemma \ref{lem:Hall} instead implies that there is a matching of size $|B_4| - \epsilon M$ between $B_4$ and $C'$. To complete the matching, we consider the set of at least $|C'| - |B_4|$ remaining vertices of $C'$  and show that if $q$ is the size of the maximum matching on this set and $2q < |C'| - |B_4| - 3\sqrt{\epsilon} M$, then $\max(\binom{2q+1}{2}, \binom{q}{2} + (|C'| - |B_4| - q)q) < \binom{|C'|-|B_4|}{2} - \epsilon\binom{M}{2}$, which would contradict Lemma \ref{ErdosGallai}.  
We may clearly assume that $|C'|-|B_4| \geq 3\sqrt{\epsilon} M$. But then 
\begin{align*}
\binom{2q+1}{2} & = (2q+1)q < (|C'| - |B_4| - 3\sqrt{\epsilon} M + 1) ( |C'| - |B_4| - 3\sqrt{\epsilon} M)/2 \\ & \leq  \binom{|C'|-|B_4|}{2} - (6 \sqrt{\epsilon} M-2)( |C'| - |B_4|)/2 + 9\epsilon M^2 /2 \\
& \leq \binom{|C'|-|B_4|}{2} -(6 \sqrt{\epsilon} M-2) (3\sqrt{\epsilon} M)/2 + 9\epsilon M^2 /2 < \binom{|C'|-|B_4|}{2} - \epsilon\binom{M}{2}
\end{align*}
and 
\begin{align*}\binom{q}{2} + (|C'| - |B_4| - q)q & \leq q(q/2 + (|C'| - |B_4| - q)) = q (|C'| - |B_4| - q/2) \\ 
& < (|C'| - |B_4| - 3\sqrt{\epsilon} M) (|C'| - |B_4|)/2 \\ & = \binom{|C'| - |B_4|}{2}  - (3\sqrt{\epsilon} M - 1) (|C'| - |B_4|)/2 \\ 
& \leq \binom{|C'| - |B_4|}{2} - (3\sqrt{\epsilon} M-1)(3\sqrt{\epsilon} M)/2 <  \binom{|C'|-|B_4|}{2} - \epsilon\binom{M}{2},
\end{align*}
as required.

As there are no red edges in $C'$ and the largest red matching between $B_4$ and $C'$ has size less than $2\mu$, there is a blue matching in $B_4 \cup C'$ with at least $|C'|+\min(|B_4|,|C'|)-5\sqrt{\epsilon} M-4\mu$ vertices. If now $A_1$ contains a blue matching of size at least $4\mu + 3 \sqrt{\epsilon} M$, then, together with the
blue matching in $B_4 \cup C'$, we get a blue matching, each vertex of distance at most two from $v$, with the total number of vertices at least 
$$8\mu+6 \sqrt{\epsilon} M + |C'|+\min(|B_4|,|C'|)-5\sqrt{\epsilon} M-4\mu \geq (2/3+\lambda)M,$$
a contradiction. Note that in the inequality we used that $|C'| \geq |C| - 2\epsilon M \geq (1/3 - \lambda - 2\epsilon)M$ and, by Claim~\ref{sclaim1}, that $|B_4| \geq |B| - 3\mu \geq (1/3 - \lambda - 4\epsilon)M - 3 \mu$. Hence, $A_1$ does not contain a blue matching of size $4\mu + 3 \sqrt{\epsilon}M$. By Lemma \ref{ErdosGallai}, it follows that $A_1$ has at most $4\mu m + 3 \sqrt{\epsilon}M m$ blue edges. The number of pairs in $A$ not in $A_1$ is also at most $|A \setminus A_1||A| \leq \mu m$, where we used (\ref{m1m}) to obtain that 
\begin{align*}
|A\setminus A_1| & \leq m - \min(m,|C'|) + \epsilon M  \leq \max(0, m - (1/3 - \lambda - 2\epsilon)M) +\epsilon M \\
& < (1/3+\lambda/2)M -(1/3 - \lambda - 2\epsilon)M + \epsilon M < \mu.
\end{align*}
Hence, there are at most $5\mu m + 3 \sqrt{\epsilon}M m$ blue edges in $A$.

We next bound the number of blue edges between $A$ and $B$. First suppose, for the sake of contradiction, that there is a blue matching between $A_1$ and $B_4$ of size $2\lambda M$. Let the set of remaining vertices in $A_1$ be $A_1'$ and the set of remaining vertices in $B_4$ be $B_5$. Pick a maximum blue matching in the union of $C'$ and $A_1' \cup B_5$.
We claim that this second matching has size at least $|C'|-2\epsilon M$. See Figure \ref{fig5} for an illustration. 

\begin{figure}[h]
    \centering
    \includegraphics[trim=0 90 0 100,clip, scale=0.45]{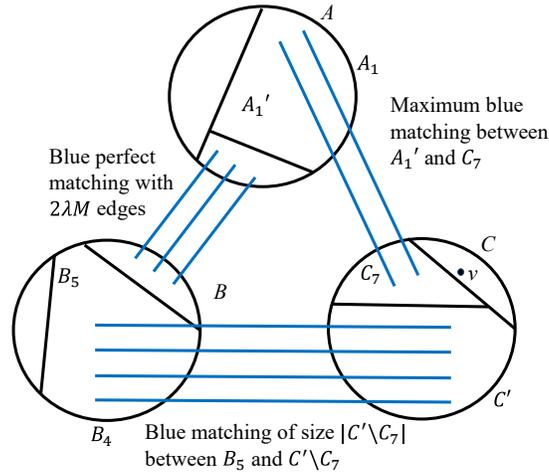}
    \caption{An illustration showing $A_1'$, $B_5$, and $C_7$ and the corresponding blue matchings.  }
    \label{fig5}
\end{figure}

To see this, we first build a blue matching between $C'$ and $B_5$. 
Since there are at most $\frac{\epsilon}{2}{M \choose 2}$ non-edges in total and there is no red matching between $B$ and $C'$ of size $2\mu$, Lemma \ref{lem:Hall} implies that the number of blue edges between $B_5$ and $C'$ is at least $|B_5| |C'| - 2\mu \max(|B_5|, |C'|) - \frac{\epsilon}{2}{M \choose 2}$. By Lemma \ref{lem:Hall} again, there is a blue matching between $B_5$ and $C'$ of size at least 
$$\frac{|B_5| |C'| -2\mu \max(|B_5|, |C'|) - \frac{\epsilon}{2}{M \choose 2}}{\max(|B_5|, |C'|)}  - 1\geq \min(|B_5|, |C'|) -2\mu - \epsilon M.$$ 
Let $C_7$ be the remaining vertices of $C'$ that are not in this matching, noting that $|C_7|$ is significantly smaller than $|A_1|$ and so also significantly smaller than $|A_1'| = |A_1|  - 2\lambda M$. Since each vertex in $A$ has red degree at most one to $C$, Lemma \ref{lem:Hall} implies that there is a blue matching between $A_1'$ and $C_7$ of size at least 
$$\frac{|A_1'||C_7| - |A_1'| - \frac{\epsilon}{2}\binom{M}{2}}{\max(|A_1'|, |C_7|)} - 1\geq |C_7| - 1 - \frac{\epsilon}{2|A_1'|}\binom{M}{2} - 1 \geq |C_7| - 2 - \epsilon M.$$
Thus, we have a matching of size at least $|C' \setminus C_7| + |C_7| - 2 - \epsilon M > |C'| - 2\epsilon M$, as required. 
Together with the matching of size $2\lambda M$ between $A_1\setminus A_1'$ and $B_4$, we see that we have a blue matching with at least 
$$4 \lambda M+2|C'|-4\epsilon M \geq (2/3+\lambda)M$$ 
vertices. But these vertices are all of distance at most two from $v$, a contradiction. Hence, there is no blue matching of size $2\lambda M$ between $A_1$ and $B_4$ and Lemma~\ref{lem:Hall} implies that there are in total at most $2\lambda M \max(|A_1|, |B_4|) \leq \lambda M^2$ blue edges between these two sets. As $|A \setminus A_1| \leq m-\min(m,|C'|)+\epsilon M$ by (\ref{m1m}) and $|B \setminus B_4| \leq 3\mu$, we see that the number of blue edges between $A$ and $B$ is at most
\begin{align*}
   &\lambda M^2  + |A\setminus A_1||B| + |B\setminus B_4||A| \\
    & \leq  \lambda M^2 + (\max(0,  m-|C'|) + \epsilon M) M + 3\mu M/2  \\
    & \leq \lambda M^2  + (\max(0, (1/3+ \lambda/2)M - (|C| - 2\epsilon M)) + \epsilon M ) M + 1.5 \mu M \\
    & \leq \lambda M^2 + ((1/3+ \lambda/2)M - (1/3  - \lambda - 2\epsilon)M +\epsilon M) M + 1.5 \mu M \\
    & \leq (1.5 \mu + 2.5\lambda M +3\epsilon M)M < 2\mu M.
\end{align*}

In total, the number of blue edges in $A \cup B$ is at most $11\mu m + 3 \sqrt{\epsilon}M m$. Hence, the number of blue edges in $\bigcup_{i \in A \cup B} V_i$ is at most 
$$\left(\frac{\epsilon}{2}{M \choose 2}+11\mu m+ 3 \sqrt{\epsilon}M m\right)\lceil n/M\rceil^2+d{n \choose 2} \leq \left(d+\frac{11\mu}{M}+ 4 \sqrt{\epsilon}\right){n \choose 2} \leq \beta {\left|\bigcup_{i \in A \cup B} V_i\right| \choose 2},$$
where we used that $m \leq M/2$, $\left|\bigcup_{i \in A \cup B} V_i\right| \geq 3n/5$ and $\beta \geq 3(d+11\mu/M +4\sqrt{\epsilon})$. 

Since the largest red matching between $B$ and $C'$ has size less than $2\mu$ and $|C \setminus C'| \leq 2\epsilon M$, it follows that the number of red edges between $B$ and $C$ is at most $\mu M$. Indeed, by Lemma \ref{lem:Hall}, the number of red edges between $B$ and $C'$ is at most $2\mu \max(|B|,|C'|)$. Thus, the number of red edges between $B$ and $C$ is at most 
\begin{align*}
2\mu \max(|B|,|C'|) + |C\setminus C'||B| &\leq2\mu \max(|B|,|C'|) + 2\epsilon M |B| \leq 2\mu \max(m, M - 2m) + 2\epsilon M|B| \\
&\leq 2\mu \max((1/3+\lambda/2)M, (1/3 + 2\lambda + 8\epsilon)M) + 2\epsilon M^2 < \mu M.
\end{align*}
Moreover, every vertex in $A$ has red degree at most one to $C$, so there are at most $|A|$ red edges between $A$ and $C$. In total, we get that the number of red edges between $\bigcup_{i \in A \cup B} V_i$ and $\bigcup_{i \in C} V_i$ is at most 
$$\left(\frac{\epsilon}{2}{M \choose 2}+\mu M+|A|\right)\lceil n/M\rceil^2+d{n \choose 2} \leq \left(d+\frac{2\mu}{M}\right)n^2 \leq \beta \left|\bigcup_{i \in A \cup C} V_i\right| \cdot \left|\bigcup_{i \in B} V_i\right|.$$
This completes the proof in this case. 
\qed






\end{document}